\documentclass[11pt,letterpaper,english,reqno]{amsart}

\usepackage{amsmath,amsfonts,amssymb,mathrsfs,stackrel}
\usepackage[english]{babel}
\usepackage[utf8x]{inputenc}
\usepackage[T1]{fontenc}
\usepackage{graphicx}
\usepackage{tabularx}
\usepackage{hyperref}
\usepackage{mathtools}
\usepackage{cleveref}
\usepackage{comment}
\usepackage[all]{xy}

\newcommand{\nocontentsline}[3]{}
\newcommand{\tocless}[2]{\bgroup\let\addcontentsline=\nocontentsline#1{#2}\egroup}
\DeclareUnicodeCharacter{0177}{\^u}

\DeclareMathOperator{\Spec}{Spec}

\DeclareMathOperator{\CH}{CH}

\DeclareMathOperator{\length}{length}

\numberwithin{equation}{subsection}

\theoremstyle{plain}

\newtheorem{theorem}{Theorem}[section]

\newtheorem{proposition}[theorem]{Proposition}
\newtheorem{lemme}[theorem]{Lemma}

\newtheorem{lemma}[theorem]{Lemma}

\newtheorem{corollaire}[theorem]{Corollary}

\theoremstyle{definition}

\theoremstyle{remark}

\newcommand{\R}{\mathbb{R}}

\newcommand{\Z}{\mathbb{Z}}

\newcommand{\Q}{\mathbb{Q}}

\newcommand{\F}{\mathbb{F}}
\newcommand{\A}{\mathbb{A}}
\newcommand{\C}{\mathbb{C}}

\newcommand{\D}{\mathcal{D}}

\newcommand{\cB}{\mathcal{B}}
\newcommand{\cO}{\mathcal{O}}
\newcommand{\cM}{\mathcal{M}}
 
\newcommand{\cY}{\mathcal{Y}}
\newcommand{\cZ}{\mathcal{Z}}

\newcommand{\scrS}{\mathscr{S}}

\newcommand{\bF}{\mathbb{F}}

\newcommand{\fP}{\mathfrak{P}}

\newcommand{\inj}{\hookrightarrow}

\usepackage[usenames,dvipsnames]{color}

\title[]{Picard rank jumps for K3 surfaces with bad reduction}

\author{Salim Tayou}
\address{Department of Mathematics, Harvard University, 1 Oxford St, Cambridge, MA 02138, USA}
\email{tayou@math.harvard.edu}
\date\today

\begin{document}

\begin{abstract}
Let $X$ be a K3 surface over a number field. We prove that $X$ has infinitely many specializations where its Picard rank jumps, hence extending our previous work with Shankar--Shankar--Tang to the case where $X$ has bad reduction. We prove a similar result for generically ordinary non-isotrivial families of K3 surfaces over curves over $\overline{\mathbb{F}}_p$ which extends previous work of Maulik--Shankar--Tang. As a consequence, we give a new proof of the ordinary Hecke orbit conjecture for orthogonal and unitary Shimura varieties. 
\end{abstract}
\thanks{}
\maketitle
\setcounter{tocdepth}{1}
\tableofcontents

\section{Introduction}
Let $X$ be a K3 surface over a number field $K$. Let $\mathscr{X}\rightarrow \mathscr{S}$ be a smooth projective model, where $\mathscr{S}\hookrightarrow \mathrm{Spec}(\mathcal{O}_K)$ is an open subset of the spectrum of the ring of integers $\mathcal{O}_K$. For every place $\mathfrak{P}$ of $\mathcal{O}_K$ with finite residual field $k(\mathfrak{P})$, we have an injective specialization map: 
\[\mathrm{Pic}(X_{\overline{K}})\hookrightarrow \mathrm{Pic}(\mathscr{X}_{\overline{k(\mathfrak{P})}}),\]
and both groups have finite rank, \emph{the Picard rank}, denoted $\rho(X_{\overline{K}})$ and $\rho(\mathscr{X}_{\overline{k(\mathfrak{P})}})$ respectively. 

Inspired by the classical density result of Noether--Lefschetz loci for weight $2$ polarized variations of Hodge structures, see \cite{voisin,oguiso}, Charles asked in \cite{charles-Picard-number} what can be said about the \emph{arithmetic Noether--Lefschetz locus}:
\[NL=\{\mathfrak{P}\in \mathscr{S} |\, \rho(X_{\overline{K}})<\rho(\mathscr{X}_{\overline{k(\mathfrak{P})}})\}.\]

In a prior work \cite[Theorem 1.1]{sstt}, we proved that the set $NL$ is infinite under the additional assumption that $X$ has potentially everywhere good reduction, i.e., up to taking a finite extension of $K$, we assumed that $\mathscr{S}=\mathrm{Spec}(\mathcal{O}_K)$. The first main result of this paper is the following unconditional result.
\begin{theorem}\label{th:K3nf}
Let $X$ be a K3 surface over a number field $K$. Then the set $NL$ is infinite.
\end{theorem}
This theorem is a particular instance of \Cref{t:main_sp_end} which is formulated for GSpin Shimura varieties and which has many other applications. As a consequence, Theorems 1.4, 1.6, and Corollary 1.7 in \cite{sstt} hold with no assumptions on the reduction type. In particular, we have the following theorem.
\begin{theorem}
Let $K$ be a number field and $A$ an abelian surface over $K$. Then there exists infinitely many places where $A$ has good reduction, and the reduction is geometrically non-simple.
\end{theorem}

\subsection{Picard rank jumps over function fields } 
Let $p\geq 5$ be a prime number. Let $\mathscr{X}\rightarrow \mathscr{S}$ be a family of K3 surfaces over a curve $\mathscr{S}$ over $\overline{\mathbb{F}}_p$. Let $\eta$ be the generic point of $\mathscr{S}$. For every $s\in\mathscr{S}(\overline{\mathbb{F}}_p)$, we have similarly an inequality of Picard ranks \[\rho(\mathscr{X}_{\overline{\eta}})\leq \rho(\mathscr{X}_{s}),\] 
and one can introduce similarly the Noether--Lefschetz locus as the subset of $\mathscr{S}$ where the above inequality is strict: 
\[NL=\{s\in \mathscr{S}(\overline{\mathbb F}_p)|\rho(\mathscr{X}_{\overline{\eta}})< \rho(\mathscr{X}_{s})\}.\] 
In \cite[Theorem 1.1]{maulik-shankar-tang-K3}, Maulik, Shankar and Tang proved that if $\mathscr{S}$ is proper and the family $\mathscr{X}\rightarrow \mathscr{S}$ is generically ordinary and not isotrivial then the set $NL$ is infinite. Our second main theorem in this paper is to remove the properness assumption in their result. 

\begin{theorem}\label{th:K3ff}
Let $\mathscr{X}\rightarrow \mathscr{S}$ be a generically ordinary non-isotrivial family of K3 surfaces over a smooth curve $\mathscr{S}$ over $\overline{\mathbb{F}}_p$  with $p\geq 5$. Suppose that the discriminant of the generic geometric Picard lattice is prime to $p$. Then the locus $NL$ is infinite.
\end{theorem}

The theorem is also a particular instance of \Cref{th:main-ff-general} for GSpin Shimura varieties, which has several other applications and also has an analogue for unitary Shimura varieties, see \Cref{unitary}. In particular, we have the following theorem which extends \cite[Theorem 1 (1)]{maulik-shankar-tang} to the quasi-projective case. 
\begin{theorem}
Let $A$ be a non-isotrivial ordinary abelian surface over the function field of a curve over $\overline{\mathbb{F}}_p$. Then $A$ has infinitely many smooth and non-simple specializations.
\end{theorem}

Both \Cref{th:K3nf} and \Cref{th:K3ff} are motivated by the density of Hodge loci in polarized variations of Hodge structure of weight $2$ of K3 type, see for example \cite{voisin,oguiso,tayouequi}. Recent density results for general polarized variations of Hodge structures of level less than $2$ as in \cite{tayoutholozan,baldi-klingler-ullmo} suggest that density of Hodge loci in arithmetic and function field settings are natural problems to investigate, and we hope to address these questions in future work. 
\subsection{Hecke orbit conjecture}
As an application of \Cref{th:K3ff}, we give a new proof of the Hecke orbit conjecture for orthogonal and certain unitary Shimura varieties. We refer to \cite[Section 1.2]{maulik-shankar-tang-K3} for the context and prior results on this conjecture.
\begin{theorem}\label{t:hecke-orbit}
Let $\mathcal{M}_{\mathbb{F}_p}$ be the reduction at $p\geq 5$ of the integral model of a Shimura variety of either: 
\begin{enumerate}
    \item Orthogonal type associated to a lattice of signature $(b,2)$ having discriminant prime to p.
    \item Unitary type associated to an imaginary quadratic field $K$ split at $p$ and to a Hermitian lattice over $\mathcal{O}_K$ of signature $(n,1)$ with discriminant prime to $p$.
\end{enumerate} 
Then the prime-to-p Hecke orbit of an ordinary point is Zariski dense in $\mathcal{M}_{\mathbb{F}_p}$.
\end{theorem}

The density of Hecke orbits in characteristic zero is a consequence of the work of Clozel--Oh--Ullmo \cite{clozelullmo}, see also \cite{eskinoh} for a dynamical approach using Ratner theory. Chai first proved the Hecke orbit conjecture for the ordinary locus of the moduli space of principally polarized abelian varieties in \cite{chai-hecke-orbit}. For orthogonal and some unitary Shimura varieties, a first proof of the Hecke orbit conjecture in the ordinary case has been obtained by Maulik--Shankar--Tang in \cite{maulik-shankar-tang-K3} and our approach is inspired from theirs. Very recently, Pol Van Hoften \cite{polvanhoften} proved this conjecture for the ordinary locus of Shimura varieties of abelian type under certain conditions on the reflex field and using completely different methods. 
\subsection{Strategy of the proof}
\Cref{th:K3nf} and \Cref{th:K3ff} are proved using a strategy initiated by Chai-Oort \cite{chai_oort} and Charles \cite{charles-exceptional-isogenies} for the product of two modular curves and subsequently used in \cite{maulik-shankar-tang,shankar-tang} for Hilbert modular surfaces over number fields and Siegel threefolds over $\overline{\F}_p$. Here we follow the set-up in \cite{sstt} and \cite{maulik-shankar-tang-K3} to which we refer for more details. For \Cref{th:K3nf}, we first translate it into  an intersection theory type statement between a curve and a sequence of divisors in the integral model of a toroidal compactification of a Shimura variety of GSpin type. For this matter, we use the Arakelov intersection theory with prelog forms developed in \cite{burgos}. We follow a similar approach for \Cref{th:K3ff}, using the usual intersection theory on the reduction modulo $p$ of the aforementioned compactification of the integral model of a GSpin Shimura variety. The new ingredients which were missing in both \cite{sstt} and \cite{maulik-shankar-tang-K3} are the local estimates on multiplicities of intersection with special divisors at points of bad reduction and the estimates of extra terms coming from the boundary divisors in the global intersection numbers coming from the work of \cite{bruinierzemel}, see also \cite{engel-greer-tayou} for a recent approach. These are the main contributions of this paper. To obtain the first estimates, we use an explicit description of the special divisors in the formal completions along toroidal boundary components. This allows us to define in each case a decreasing sequence of positive definite lattices $(L_n,Q)$ which computes the local intersection number. We give an estimate on the growth of the successive minima of these lattices, then a geometry-of-numbers type argument allows us to derive the desired estimates. To obtain the bounds on the extra terms in the global intersection number, we use the explicit expressions from \cite{bruinierzemel} and \cite{bruinier} combined with an equidistribution result from \cite{duke-invetiones-hyperbolic,eskinoh}. 
\subsection{Organization of the paper}
The key input of this paper is the description of the special divisors in terms of local coordinates of integral models of toroidal compactifications of Shimura varieties of GSpin type. In Section 2, we explain  these constructions following \cite{howardmadapusi} and \cite{madapusiintegral}, and the section culminates with a description of the special divisors in formal completions along locally closed boundary divisors. In Section 3, we recall briefly Arakelov arithmetic intersection theory with prelog forms following \cite{burgos}, and we assemble different ingredients from the literature (\cite{bruinierzemel,howardmadapusi,borcherdszagier}) to state the modularity of the generating series of special divisors in the integral models of toroidal compactifications of Shimura varieties of GSpin type. In Section 4, we state the archimidean and finite place estimates needed to prove our main theorems,  and then we prove the archimidean estimates. Section 5 is devoted to estimating contributions from bad reduction places. Finally, we prove the application to Hecke orbit conjecture in Section 6. 

\subsection{Acknowledgments} I am very grateful to Fran\c{c}ois Charles, Fran\c{c}ois Greer, Keerthi Madapusi Pera, Ananth Shankar and Yunqing Tang for many helpful discussions and insights. This project started while I was invited by Yunqing Tang to Princeton University, and I also benefited from support from the Institute for Advanced Study in Princeton. I thank both institutions for their hospitality. I thank the referee for the very valuable comments and suggestions.

\section{GSpin Shimura varieties: integral models and compactifications}
This section summarizes the construction of the GSpin Shimura variety, its toroidal compactifications and their integral models following \cite{bruinierzemel,howardmadapusi,madapusitor,agmp-annals}, see also \cite{kisin,madapusiintegral,pink} for earlier work. The ultimate goal is to describe the special divisors in formal completions along the toroidal boundary strata. The familiar reader may wish to skip directly to \Cref{s:special-div-main} for these results. 
\subsection{The GSpin Shimura variety}\label{ss:gspinshimura}
Let $(L,Q)$ be an even quadratic lattice of signature $(b,2)$ with $b\geq 1$ and with associated even bilinear form \[(\,.\,):L\times L\rightarrow \Z\] such that $Q(x)=\frac{(x.x)}{2}\in\Z$ for all $x\in L$. 

Let $G=\mathrm{GSpin}(L_\Q)$ be the algebraic group over $\Q$ of spinor similitudes defined as in \cite[Section 1.2]{madapusiintegral}. The group $G(\R)$ acts on the Hermitian symmetric space
\[\mathcal{D}=:\{z\in \mathbb{P}(L_\C)|\,(z.z)=0,\,(z.\overline{z})<0\}.\]
The pair $(G,\mathcal{D})$ is the {\it GSpin Shimura datum}. Its reflex field is $\Q$ by \cite[Section 3.1]{madapusiintegral}.
\medskip

For $K\subset G(\mathbb{A}_f)$ a compact open subgroup, the {\it GSpin Shimura variety} 
\[M(\C)=G(\Q)\backslash \D\times G(\mathbb{A}_f)/K\]
is the set of complex points of a Deligne--Mumford stack $M$ defined over $\Q$. In what follows, we  choose the compact open group $K\subset G(\mathbb{A}_f)$ as in \cite[Equation (4.1.2)]{agmp-annals}. Its image in $\mathrm{SO}(L_\Q)(\mathbb{A}_f)$ stabilizes $L\otimes \widehat{\Z}\subseteq L\otimes \mathbb{A}_f$ and is equal to the subgroup that acts trivially on the quotient $\widehat{L}^{\vee}/\widehat{L}=L^{\vee}/L$, where the dual lattice $L^{\vee}$ is defined as 
\[L^{\vee}=\{x\in L_\Q| \forall y\in L,\, (x.y)\in\Z\}.\]

The Shimura variety $M$ carries a  line bundle of weight $1$ modular forms that we denote by $\mathcal{L}_\Q$ and we refer to \cite[Section 4.1]{agmp-annals} for a definition. The Shimura datum $(G,\mathcal{D})$ is of Hodge type by \cite[Section 2.2]{agmp-compositio}: there exists a Shimura datum of Siegel type $(G^{Sg},\D^{Sg})$ and a compact open subgroup $K^{sg}\subset G^{sg}(\mathbb{A}_f)$ such that we have an embedding of Shimura varieties over $\Q$
\[M\hookrightarrow M^{Sg}.\]
This is the {\it Kuga--Satake embedding}. The pull-back of the universal abelian scheme on $M^{Sg}$ yields the {\it Kuga--Satake} abelian scheme $A\rightarrow M$.

\subsection{Toroidal compactifications over $\C$}In this section, we describe the toroidal compactifcations of $M$ as well as the structure of the boundary components following \cite{bruinierzemel} and \cite{howardmadapusi}. See also \cite{amrt} for the general theory of toroidal compactifications over $\C$.
\medskip 

Recall from \cite[Section 2.2]{howardmadapusi} that an admissible parabolic subgroup $P\subseteq G$ is either a maximal proper parabolic subgroup of $G$ or $G$ itself. \footnote{$G^{ad}$ is simple in our case.}
A {\it cusp label representative} $\Phi=(P,\D^{\circ},h)$ is a triple constituted from an admissible parabolic subgroup $P$, a connected component $\D^{\circ}\subset \D$ and an element $h\in G(\mathbb{A}_f)$. 

Attached to a cusp label representative $\Phi=(P,\D^{\circ},h)$, there exists a mixed Shimura variety that we now describe. Let $U_{\Phi}$ be the unipotent radical of $P$ and let $W_{\Phi}$ be the center of $U_{\Phi}$ \footnote{We follow the notations of \cite{madapusitor} which differ from other references.}. Let $Q_\Phi$ be the normal subgroup of $P$ defined as in \cite[\S 4.7]{pink}, see also \cite[2.2]{howardmadapusi}. Define as in {\it loc. cit.} $\D_\Phi=Q_\Phi(\R)W_\Phi(\C)\D^{\circ}$ and let $K_\Phi=hKh^{-1}\cap Q_{\Phi}(\mathbb{A}_f)$. 
We define then the mixed Shimura variety 
\begin{align}\label{def-mixed}
    M_{\Phi}(\C)=Q_\Phi(\Q)\backslash \D_\Phi\times Q_\Phi(\mathbb{A}_f)/K_\Phi.
\end{align}
By \cite[Prop. 12.1]{pink}, $M^{\Phi}(\C)$ has a canonical model $M^{\Phi}$ also defined over $\Q$. Let $\overline{Q}_\Phi=Q_\Phi/W_\Phi$ and $\overline{\D}_\Phi=W_\Phi(\C)\backslash \D_\Phi$. Let $\overline{K}_\Phi$ be the image of $K_\Phi$ under the quotient map  $Q_\Phi(\mathbb{A}_f)\rightarrow \overline{Q}_\Phi(\mathbb{A}_f)$. Then from the data $(\overline{Q}_\Phi,\overline{\D}_\Phi,\overline{K}_\Phi)$ we define similarly to \ref{def-mixed} a mixed Shimura variety $\overline{M}_\phi$ and we have a canonical morphism  
\begin{align}\label{eq:torus-fibration}
    M_{\Phi}\rightarrow \overline{M}_\Phi.
\end{align}
This map has a torsor structure that we now describe. Let $\Gamma_\Phi=K_\Phi\cap W_\Phi(\Q)$. It is a $\Z$-lattice in $W_\Phi(\Q)$. By \cite[Proposition 2.3.1]{howardmadapusi}, the map  \ref{eq:torus-fibration}  is canonically a torsor under the torus $T_{\Phi,\Q}$ whose cocharacter group is $\Gamma_\Phi$.

The mixed Shimura variety $\overline{M}_{\Phi}$ has itself a fibration structure over a pure Shimura variety constructed as follows, see \cite[2.1.7]{madapusitor} for more details. 
\medskip

Let $G^{h}_{\Phi}=Q_\Phi/U_\Phi$ be the Levi quotient of $Q_\Phi$, $V_\Phi=U_\Phi/W_\Phi$ the unipotent radical of $\overline{Q}_\Phi$ and let $\D^{h}_{\Phi}=V_\Phi(\R)\backslash \overline{D}_\Phi$. Then the pair $(G^{h}_{\Phi},\D^{h}_{\Phi})$ is a pure Shimura datum with reflex field equal to $\Q$. Let $K_{\Phi}^{h}\subset G_\Phi^{h}(\mathbb{A}_f)$ be the image of $K_\Phi$. Then the quotient
\[M_{\Phi}^{h}(\C)=G^{h}_{\Phi}(\Q)\backslash (\D^{h}_{\Phi}\times G^{h}_{\Phi}(\mathbb{A}_f))/K_{\Phi}^{h}\]
is the set of complex points of a Shimura variety which admits a canonical model $M_{\Phi}^{h}$ defined over $\Q$ and we have a canonical map 
\begin{align}\label{map-torsor-structure}
    \overline{M}_{\Phi}\rightarrow M_{\Phi}^{h}.
\end{align}
By \cite[2.1.12]{madapusitor}, there exists a natural abelian scheme $A_K(\Phi)\rightarrow M_{\Phi}^{h}$ such that the map \ref{map-torsor-structure} is a torsor under $A_K(\Phi)$.
\medskip 

In what follows, we will describe the above data for the GSpin Shimura variety introduced in \Cref{ss:gspinshimura} following \cite[Section 4]{howardmadapusi} and \cite[Section 3]{bruinierzemel}. Let $\Phi$ be a cusp label representative. The admissible parabolic subgroup $P$ is the stabilizer of a totally isotropic subspace $I_\Phi$ of $L_\Q$ of dimension at most $2$. The dimension $0$ case corresponds to $P=G$. If $P$ is the stabilizer of a primitive isotropic line $I_\Q\subset L_\Q$, then the cusp label representative is said to be of type III. If $P$ is the stabilizer of a primitive isotropic plane $J_\Q\subset L_\Q$, then $\Phi$ is said to be of type III. We will follow the notations of \cite{bruinierzemel} and denote by $\Upsilon$, resp. $\Xi$, a cusp label representative of type II, resp. of type III. 
\medskip 

Given two cusp label representatives $\Phi_1$ and $\Phi_2$, there is a notion of a $K$-morphism $\Phi_1\xrightarrow{(\gamma,q_2)_K} \Phi_2$ given by $\gamma\in G(\Q)$ and $q_2\in Q_{\Phi_2}(\mathbb{A}_f)$ which we don't define here and refer to \cite[2.1.14]{madapusitor} for the definition, see also \cite[Definition 2.4.1]{howardmadapusi}. 
\medskip

Let $\Phi$ be a cusp label representative. By the general theory of toroidal compactifications, see \cite[4.15]{pink} or \cite[Ch. II.\S 1.1]{amrt} for the definitions, there exists a canonical open non-degenerate self-adjoint convex cone $C_\Phi\subset W_\Phi(\R)$ homogeneous under $P(\R)$ and which allows to realize $\mathcal{D}^{\circ}$ as a tube domain inside an affine space, see \cite[2.1.5]{madapusitor}. We define the {\it extended cone} $C^{*}_\Phi$ as in \cite[2.1.22]{madapusitor}: for any map $\Phi'\xrightarrow{(\gamma,q)_K} \Phi$, the conjugation by $\gamma^{-1}$ induces an embedding \[\mathrm{int}(\gamma^{-1}):W_{\Phi'}(\R)\hookrightarrow W_{\Phi}(\R)\] and we define then 
\[ C_{\Phi}^{*}=\bigcup_{\Phi'\rightarrow \Phi}\mathrm{int}(C_{\Phi'}).\]
This cone lies between $C_\Phi$ and its topological closure in $W_\Phi(\R)$ but in general, it is neither open nor closed. See also \cite[Definition-Proposition 4.22]{pink} for more details.
\medskip

Recall from \cite[Definition 2.4.3]{howardmadapusi} that a {\it rational polyhedral cone decomposition} (rpcd for short) of $C^{*}_{\Phi}$ is a collection $\Sigma_{\Phi}=\{\sigma\}$ of rational polyhedral cones $\sigma\subset W_\Phi(\R)$ satisfying natural compatibility conditions (we don't recall these conditions here and invite the reader to consult the reference above for more information). The rpcd $\Sigma_\Phi$ is said to be smooth if it is smooth in the sense of \cite[\S 5.2]{pink} with respect to the lattice $\Gamma_\Phi$. It is complete if \[C^*_{\Phi}=\bigcup_{\sigma\in \Sigma_\Phi} \sigma.\]

\subsubsection{Boundary components of type II}\label{s:typeII}
Let $\Upsilon$ be a cusp label representative of type II. Then $P$ is the stabilizer of a primitive isotropic plane $J_\Q \subseteq L_\Q$ and let $J=J_\Q\cap h.L$, where $h\in G(\mathbb{A}_f)$ acts on $L$ via the map $G(\mathbb{A}_f)\rightarrow \mathrm{SO}(L_\Q)(\mathbb{A}_f)$. 
Then by \cite[Page 31]{howardmadapusi}, the group $W_\Upsilon$ is identified with $\bigwedge^{2}J_\Q$, hence it is one dimensional. The lattice $\Gamma_\Upsilon\subset W_\Upsilon$ is also of rank $1$. The open convex cone $C_{\Upsilon}$ is given by a half line $\R^{+}\backslash\{0\}$ and the extended cone is $C^{*}_{\Upsilon}=\{0\}\cup C_\Upsilon$. 
\medskip 

Let $M_{\Upsilon}$ and $\overline{M}_{\Upsilon}$ be the  mixed Shimura varieties associated to $\Upsilon$. Then $M_{\Upsilon}\rightarrow \overline M_{\Upsilon}$ is a torsor under the $1$ dimensional torus $T_\Upsilon$ with cocharacter group $\Gamma_\Upsilon$. The group $G^h_{\Upsilon}$ is equal to $\mathrm{SL}_2$ and $\D^h_{\Upsilon}$ is equal to the Poincar\'e upper half-plane. The Shimura variety $M^{h}_\Upsilon$ is a modular curve and the abelian scheme $A_\Upsilon$ is equal to the Kuga--Sato variety $D\otimes E$ where $E\rightarrow M^{h}_{\Upsilon}$ is the universal elliptic curve over over $M^{h}_{\Upsilon}$ and $D$ is the positive definite plane $J^{\bot}/J$,  see \cite[Corollary 3.17]{bruinierzemel}  and \cite[Proposition 4.3]{zemel} for details and proofs. Notice that our choice for the compact open subgroup $K$ gives exactly the stable orthogonal group used in \cite{bruinierzemel} and \cite{zemel}. 
\medskip 

The only possible cone decomposition of $C^{*}_\Upsilon$ in this situation is $\Sigma_\Upsilon=\{\{0\},C_{\Upsilon}\cup \{0\}\}$ and this determines a partial compactification $M_{\Upsilon}\hookrightarrow M_{\Upsilon,\Sigma}$ which is a fibration by $\A^1_{\C}$ over $\overline M_{\Upsilon}$. Finally, there is only one boundary divisor denoted by $B_\Upsilon$ associated to the ray $C_{\Upsilon}$. 
\subsubsection{Boundary components of type III}\label{s:typeIII}
Let $\Xi$ be a cusp label representative of type III. Then $P$ is the stabilizer of a primitive isotropic line $I_\Q \subset L_\Q$ and let $I=I_\Q\cap h.L$. Set $K_I=I^{\bot}/I$. Then by \cite[Equation (4.4.2)]{howardmadapusi}, we have $U_\Xi=W_\Xi$ and we have an isomorphism of vector spaces \[K_{I,\Q}\otimes I_\Q\simeq W_{\Xi}(\Q).\]
The lattice $(K_I,Q)$ is a Lorentzian lattice of signature $(b-1,1)$. Under the above isomorphism, and assuming we have chosen a primitive generator of $I$, the open convex cone $C_\Xi\subset W_{\Xi}(\R)$, see \cite[Section 2.4]{howardmadapusi} is identified with a connected component of the light cone 
\[\{x\in K_{I,\R},\, Q(x)< 0\}.\]
The spaces $M^h_{\Xi}$ and $\overline{M}_\Xi$ are equal and are Shimura varieties of dimension zero that we can describe as follows. Let  $(\mathbb{G}_m,\mathcal{H}_0)$ be the Shimura data given by 
\[\mathcal{H}_0:=\{2\pi\epsilon: \epsilon^2=-1\},\]
on which $\R^{\times}$ acts naturally through the quotient $\R^{\times}/\R^{\times}_+$. There is a morphism of mixed Shimura data $(Q_\Xi,\D_\Xi)\rightarrow (\mathbb{G}_m,\mathcal{H}_0)$ given by a canonical character $v_\Xi:Q_\Xi\rightarrow \mathbb{G}_m$ defined as in \cite[Equation (4.4.1)]{howardmadapusi} and a map $\D_\Xi\rightarrow \mathcal{H}_0$ given as in \cite[Equation (4.6.3)]{howardmadapusi}. Then the Shimura variety $\mathrm{Sh}_{\nu_\Xi(K_\Xi)}(\mathbb{G}_m,\mathcal{H}_0)$ is zero dimensional and the canonical map $M_\Xi\rightarrow \mathrm{Sh}_{\nu_\Xi(K_\Xi)}(\mathbb{G}_m,\mathcal{H}_0)$ is a torsor under the torus $T_\Xi=\mathrm{Spec}\left(\Q[q_\alpha]_{\alpha\in\Gamma_\Xi^\vee}\right)$ with cocharacter group $\Gamma_\Xi=K_I$ by \cite[Proposition 3.7]{bruinierzemel}.
\medskip 

The intermediate cone  $C_{\Phi}^{*}$ can be described explicitly as follows, see also \cite[page 23]{bruinierzemel}: for any type II boundary component $\Upsilon$ with corresponding isotropic plane $J$ containing $I$, the quotient $J/I$ has a generator $\omega_{\Xi,\Upsilon}$ lying on the boundary of the $C_\Xi$. Hence \[C_{\Xi}^{*}=C_{\Xi}\cup\bigcup_{\Upsilon}\R\omega_{\Xi,\Upsilon}.\]
The rays $\R\omega_{\Xi,\Upsilon}$ will be referred to as the external rays and the rays in $C_{\Xi}$ are the {\it inner rays}.

\subsubsection{Toroidal compactifications}
Recall from \cite[Definition 2.4.4]{howardmadapusi} that a $K$-admissible rational polyhedral cone decomposition for $(G,\D)$ is a collection   $\Sigma=\{\Sigma_\Xi,\Sigma_{\Upsilon}\}$ such that  $\Sigma_\Xi$ and $\Sigma_\Upsilon$ are rpcd for any cusp label representative $\Xi$ and $\Upsilon$ respectively satisfying the compatibility conditions of \cite[Definitions 2.4.3, 2.4.4]{howardmadapusi}. It is said smooth (resp. complete) if every $\Sigma_\Phi$ is smooth (resp. complete).  

A toroidal stratum representative is a pair $(\Phi,\sigma)$ where $\Phi$ is a cusp label representative and $\sigma\subset C_{\Phi}^{*}$ is a rational polyhedral cone whose interior is contained in $C_\Phi$. 
There is similarly a notion of $K$-morphism between stratum representatives, see \cite[Definition 2.4.6]{howardmadapusi} and the set of $K$-isomorphism classes of toroidal stratum representatives will be denoted $\mathrm{Start}_K(G,\D,\Sigma)$.
We say that $\Sigma$ is finite if \[|\mathrm{Start}_K(G,\D,\Sigma)|<\infty.\]

Let $\Sigma$ be a finite $K$-admissible complete cone decomposition. The main result of \cite[\S 12]{pink}, see also \cite[Theorem 2.1.27]{madapusitor}, ensures that there exists a proper toroidal compactification \[M\hookrightarrow M^{\Sigma}\] in the category of Deligne-Mumford stacks over $\Q$ such that $M^{\Sigma}$ is proper over $\Q$ and has a stratification 
\begin{align}\label{eq:stratification-generic}
    M^{\Sigma}=\bigsqcup_{(\Phi,\sigma)\in\mathrm{Start}_K(G,\D,\Sigma)} B^{\Phi,\sigma}
\end{align}
by locally closed subspaces indexed by the finite set of strata $\mathrm{Start}_K(G,\D,\Sigma)$. The stratum indexed by $(\Phi,\sigma)$ lies in the closure of the stratum index by $(\Phi',\sigma')$ if and only if there is a $K$-morphism of strata representatives $(\Phi,\sigma)\rightarrow (\Phi',\sigma')$. Then the closure of the stratum $B^{\Phi,\sigma}_K$ is given by 
\[\overline{B^{\Phi,\sigma}}=\bigcup_{(\Phi',\sigma')\rightarrow (\Phi,\sigma)}B^{\Phi',\sigma'}.\]

Moreover, by \cite[Theorem 3.4.1]{howardmadapusi} following the work of Harris and Harris--Zucker \cite{harriszucker}, the line bundle of weight $1$ modular forms $\mathcal{L}$ extends to a line bundle on $M^{\Sigma}$ which we still denote $\mathcal{L}$ by abuse of notation.
\medskip 

Let $(\Phi,\sigma)$ be a toroidal stratum representative. Then $(\Phi,\sigma)$ determines a partial compactification of the mixed Shimura variety $M_\Phi\hookrightarrow M_\Phi(\sigma)$ with boundary component index by $\sigma$ denoted by $Z^{\Phi}(\sigma)$. Pink proved that there is a canonical isomorphism \cite[Corollary 7.17, Theorem 12.4]{pink}, see also \cite[Theorem 2.1.27]{madapusitor}, of Deligne--Mumford stacks: 
\[\Delta_K(\Phi,\sigma)\backslash Z^{\sigma}(\sigma)\simeq B^{\Phi,\sigma}\] 
where $\Delta_K(\Phi,\sigma)$ is the finite group defined in \cite[2.1.19]{madapusitor}. The latter induces an isomorphism of formal Deligne--Mumford stacks:
\begin{align}\label{eq:iso-completions}
\Delta_K(\Phi,\sigma)\backslash\widehat{M}_\Phi(\sigma)\simeq \widehat{M}^{\Sigma}
\end{align}
where $\widehat{M}_\Phi(\sigma)$ is the completion of $M_\Phi(\sigma)$ along the locally closed subspace $Z^{\Phi}(\sigma)$ and $\widehat{M}^{\Sigma}$ is the formal completion of $M^{\Sigma}$ along the locally closed stratum $B^{\Phi,\sigma}$. 

Our goal in the next two sections is to make the above isomorphisms explicit for type II and type III boundary strata.

\subsubsection{Formal completion along type II boundary strata}\label{ss:formal-tII} Let $\Upsilon$ be a cusp label representative of type II. By the discussion in \Cref{s:typeII}, there is a unique choice of a 1 dimensional ray $\sigma$ and hence  a unique choice of boundary stratum representative $(\Upsilon,\sigma)$  which corresponds to a locally closed divisor $B^{\Upsilon,\sigma}$.

The morphism $M_\Upsilon\rightarrow \overline{M}_\Upsilon$ is then a torsor under a $1$-dimensional torus $T_\Upsilon$ with cocharacter group $\Gamma_\Upsilon\simeq \Z$, i.e., $T_\Upsilon\simeq\mathrm{Spec}(\Q[q,q^{-1}])$. The partial compactification $T_\Upsilon(\sigma)$ is then isomorphic to $\mathrm{Spec}\left(\Q[q]\right)$  and the partial toroidal compactification of $M_\Upsilon$ is given as a twisted torus embedding over $\overline{M}_\Upsilon$ with fiber $\mathrm{Spec}\left(\Q[q]\right)$. Hence we have the following description of $\widehat{M}_\Upsilon(\sigma)$ 

\[\widehat{M}_{\Upsilon}(\Sigma)\xrightarrow{\mathrm{Spf}(\Q[[X]])} \overline{M}_\Upsilon\xrightarrow{D\otimes E} M^{h}_\Upsilon.\]

\subsubsection{Formal completion along type III boundary strata}\label{ss:fomral-tIII}

Let $(\Xi,\sigma)$ be a toroidal stratum representative of type III such that $\sigma$ is a one dimensional inner ray. The corresponding boundary component is denoted by $B^{\Phi,\sigma}$ and is a locally closed divisor. Write $\sigma=\R\omega$ where $\omega\in C_\Xi\cap K$ is an integral primitive generator that satisfies $(\omega.\omega)<0$. 

The morphism $M_\Xi\rightarrow \mathrm{Sh}_{\nu_\Xi(K_\Xi)}(\mathbb{G}_m,\mathcal{H}_0)$ is a torsor under the torus \[T_\Xi=\mathrm{Spec}\left(\Q[q_\alpha]_{\alpha\in\Gamma_\Xi^\vee}\right).\]
The partial compactification $T_\Xi(\sigma)$ is equal to \[T_\Xi(\sigma)=\mathrm{Spec}\left(\Q[q_\alpha]_{\underset{(\alpha.\omega)\geq 0}{\alpha\in\Gamma_\Xi^\vee}}\right)\]
and the ideal defining the boundary divisor is given by $I_\sigma=(q_\alpha,(\alpha,\omega)>0)$. It is generated by $q_{\omega'}$ for any $\omega'\in \Gamma_\Xi^\vee$ for which $(\omega,\omega')=1$.  We fix such $\omega'$. 

The formal completion along the boundary divisor is then given by: 
\[\widehat{T}_\Xi(\sigma)=\mathrm{Spec}\left(\Q[q_\alpha,\,\alpha\in\Gamma_\Xi^\vee\cap \omega^{\bot}][[q_{\omega'}]]\right),\]
and the map $M_\Xi(\sigma)\rightarrow \mathrm{Sh}_{\nu_\Xi(K_\Xi)}(\mathbb{G}_m,\mathcal{H}_0)$
is a twisted torus embedding with fibers $\widehat{T}_\Xi(\sigma)$. We will trivialize this fibration following an approach similar to \cite[Page 34]{howardmadapusi}.

First choose an auxiliary isotropic line $I_*\subset L_\Q$ such that $(I.I_*)\neq 0$. Then by \cite[Equation (4.6.6)]{howardmadapusi} and the discussion that follows, this determines a section 
\[(\mathbb{G}_m,\mathcal H_0)\xrightarrow{s} (Q_\Xi,\mathcal D_\Xi).\]

The section $s$ determines a Levi decomposition $Q_\Xi=\mathbb G_m \ltimes U_\Xi$. 
Let $K_0\subset \mathbb{G}_m(\mathbb{A}_f)$ be a compact open subgroup small enough such that the image under the section $s$ is contained in $K_\Xi$ and let 
\[K_{\Xi,0}=K_0\ltimes (U_\Xi(\mathbb{A}_f\cap K_\Xi))\subset K_\Xi.\]

Then by reasoning similarly to \cite[Proposition 4.6.2]{howardmadapusi}, we have the following. 
\begin{proposition}\label{p:local-chart}
We have an isomorphism of formal algebraic spaces: 
\[\underset{{ a\in  \Q^\times_{>0} \backslash \mathbb{A}_f^\times / K_0  }}{\bigsqcup} \widehat{T}_\Xi(\sigma)_{/\C} \xrightarrow{\simeq} \widehat{M}_{K_{\Xi,0}}   (\sigma )_{/\C}, \]
and the map  \[\widehat{M}_{K_{\Xi,0}}(\sigma)_{/\C}\rightarrow \widehat{M}_{K_{\Xi}}(\sigma)_{/\C}\]
is a formally \'etale map of formal Deligne--Mumford stacks given by the quotient by $K_\Xi/K_{\Xi,0}$. 
In particular, if $K$ is neat, then the above map is a formally  \'etale surjection of algebraic spaces.
\end{proposition}
\begin{proof}
The same proof as in \cite[Proposition 4.6.2]{howardmadapusi} works with no change in our setting.
\end{proof}

\subsection{Integral models}\label{s:integral-models}
We recall in this section the construction of integral models of GSpin Shimura varieties and their compactifications following \cite{howardmadapusi,agmp-annals,madapusitor}. We  assume henceforth that the lattice $(L,Q)$ is a maximal lattice, i.e., there is no strict superlattice in $L_\Q$ containing $L$ over which $Q$ is $\Z$-valued.
\medskip

By \cite[Section 4.4]{agmp-annals}, there exists a flat and normal integral model $\mathcal{M}\rightarrow \mathrm{Spec}(\Z)$ which is a Deligne--Mumford stack of finite type over $\Z$. It enjoys the following properties:
\begin{enumerate}
    \item If the lattice $(L,Q)$ is almost self dual\footnote{See \cite[Definition 6.1.1]{howardmadapusi}.} at a prime $p$ then the restriction of the integral model to $\mathrm{Spec}(\Z_{(p)})$ is smooth.
    \item If $p$ is odd and $p^2$ does not divide the discriminant of $(L,Q)$, the restriction of $\mathcal{M}$ to $\mathrm{Spec}(\Z_{(p)})$ is regular. 
    \item If $n\geq 6$, the reduction mod $p$ is geometrically normal. 
    \item The line bundle of modular forms of weight $1$ extends to a line bundle on $\mathcal{M}$ that we denote by $\mathcal{L}$.
\end{enumerate}

Furthermore, given a $K$-admissible polyhedral complete cone decomposition,  $\mathcal{M}$ admits by \cite[Theorem 4.1.5]{madapusitor} a toroidal compactification $\mathcal{M}^{\Sigma}$ proper over $\mathrm{Spec}(\Z)$ and which extends the compactification $M^{\Sigma}$ previously defined over $\Q$. Moreover, it has a stratification  \begin{align}\label{eq:stratification-integral}
    \mathcal{M}^{\Sigma}=\bigsqcup_{(\Phi,\sigma)\in\mathrm{Start}_K(G,\D,\Sigma)} \mathcal{B}^{\Phi,\sigma}
\end{align}
which extends the stratification in  \Cref{eq:stratification-generic} and such every stratum is flat over $\Z$.
The unique open stratum is $\mathcal{M}$ and its complement is a Cartier divisor. Moreover, for any cusp label representative $(\Phi,\sigma)$, the tower of maps \[M_\Phi(\sigma)\rightarrow \overline{M}_\phi\rightarrow M^h_\Phi\]
has an integral model 
\[\cM_\Phi(\sigma)\rightarrow \overline{\cM}_\phi\rightarrow \cM^h_\Phi\]
which satisfies the following: the abelian scheme $A_\Phi$ has an extension $\mathcal{A}_\Phi\rightarrow \cM^h_\Phi$ such that the map $\overline{\cM}_\phi\rightarrow \cM^h_\Phi$ is a torsor under $\mathcal{A}_\Phi$ and the map $\cM_\Phi(\sigma)\rightarrow \overline{\cM}_\phi$ is a twisted torus embedding with structure group the torsor $\mathcal{T}_\Xi $ extending $T_\Xi$. Finally, the boundary component $Z_\Phi(\sigma)$ has a flat extension $\mathcal{Z}_\Phi(\sigma)$ such that we have an isomorphism of completions:
\begin{align}\label{eq:iso-comp-int}
\Delta_K(\Phi,\sigma)\backslash\widehat{\cM}_\Phi(\sigma)\simeq \widehat{\cM}^{\Sigma}
\end{align}
extending the isomorphism in \Cref{eq:iso-completions}. See \cite[Theorem 4.1.5]{madapusitor} and \cite[Section 8.1]{howardmadapusi} for more details. 

Fix a prime $p$. The goal of the next two subsections is to describe the formal completions of $\mathcal{M}^{\Sigma}$ along the boundary divisors of these compactifications explicitly over $\Z_{(p)}$ in the type II and the type III case.
\subsubsection{Type II}\label{s:typeII-integral}
Let $(\Upsilon,\sigma)$ be a toroidal stratum representative of type II where $\sigma$ is the unique one dimensional ray. 

Let $\mathcal{T}_\Upsilon=\mathrm{Spec}\left(\Z_{(p)}[q,q^{-1}]\right)$ with partial compactification $\mathcal{T}_\Upsilon(\sigma)= \mathrm{Spec}\left(\Z_{(p)}[q]\right)$. By \Cref{eq:stratification-integral} and \cite[Theorem 4.1.5 (2-4)]{madapusitor}, the morphism $\mathcal{M}_\Upsilon\rightarrow \overline{\mathcal{M}}_\Upsilon$ is a torsor under $\mathcal{T}_\Upsilon$ and the morphism $ \overline{\mathcal{M}}_\Upsilon\rightarrow \mathcal{M}^h_\Xi$ is a torsor under $D\otimes\mathcal{E}$ where $\mathcal{E}\rightarrow \mathcal{M}^h_\Xi$ is the universal elliptic curve. Moreover, the partial toroidal compactification of $\mathcal{M}_\Xi$ is given as a twisted torus embedding over $\overline{\mathcal{M}}_\Upsilon$ with fibers isomorphic to $\mathcal{T}_\Upsilon(\sigma)$. In particular, the formal completion of $\mathcal{M}_{\Upsilon}$ along the boundary component is describe by the following diagram: 
\begin{align}\label{diagram-typeII}
\widehat{\mathcal{M}}_\Upsilon(\sigma)\xrightarrow{\widehat{\mathcal{T}}_\Upsilon(\sigma)} \overline{\mathcal{M}}_\Upsilon\xrightarrow{D\otimes\mathcal{E}} \mathcal{M}^h_\Upsilon,
\end{align}
where $\widehat{\mathcal{T}}_\Upsilon(\sigma)=\mathrm{Spf}\left(\Z_{(p)}[[q]]\right)$. 

\subsubsection{Type III}\label{s:typeIII-integral}
Let $(\Xi,\sigma)$ be a toroidal stratum representative of type III  such that $\sigma$ is one dimensional and generated by a primitive integral element $\omega\in C_\Xi$ with $(\omega.\omega)=-2N$. Let $\mathcal{T}_\Xi=\mathrm{Spec}\left(\Z_{(p)}[q_\alpha]_{\alpha\in\Gamma_\Xi^{\vee}}\right)$ and recall that we have a $T_\Xi$ torsor structure
\[M_\Xi\rightarrow \mathrm{Sh}_{\nu_\Xi(K_\Xi)}(\mathbb{G}_m,\mathcal{H}_0).\]

The cone $\sigma$ determines a  partial compactification $\mathcal{T}_\Xi(\sigma)=\mathrm{Spec}\left(\Z_{(p)}[q_\alpha]_{\underset{(\alpha,\omega)\geq 0}{\alpha\in\Gamma_\Xi^{\vee}}}\right)$ and also a partial compactification $\mathcal{M}_\Xi\hookrightarrow \mathcal{M}_\Xi(\sigma)$ which is a twisted torus embedding with fibers $\mathcal{T}_\Xi(\sigma)$.

The boundary divisor in $\mathcal{T}_\Xi(\sigma)$ is defined by the ideal $I_\sigma=(q_\alpha, (\alpha,\omega)>0)$. If $\omega'\in\Gamma^\vee_\Xi$ is as before an element such that $(\omega'.\omega)=1$, then $I_\sigma=(q_{\omega'})$. The formal completion of $\mathcal{T}_{\Xi}(\sigma)$ along $I_\sigma$ is then given by

\[\widehat{\mathcal{T}}_\Xi=\mathrm{Spf}\left(\Z_{(p)}[q_\alpha,\, \alpha\in\Gamma_\Xi^{\vee}\cap \omega^{\bot}][[q_{\omega'}]]\right).\]

Recall that we have a morphism of Shimura data 
\[(Q_\Xi,\D_\Xi)\xrightarrow{v_\Xi}(\mathbb{G}_m,\mathcal{H}_0),\]
and let $s$ be the section of $v_\Xi$ defined in \Cref{ss:fomral-tIII}. Let $K_0\subset \mathbb A_f^{\times}$ be a compact open subgroup such that $s(K_0)\subset K_\Xi$. We can furthermore assume that $K_0$ factors as \[K_0=\Z_p^{\times}.K_0^p.\]
Let $F$ be the abelian extension of $\Q$ determined by the reciprocity morphism in global class field theory:  
\[\mathrm{rec}: \Q^{\times}_{>0}\backslash \mathbb{A}_f^{\times}/K_0\simeq \mathrm{Gal}(F/\Q).\]
Fix a prime $\mathfrak{P}\subset \mathcal{O}_F$ above $p$ and let $R$ be the localization of $\mathcal{O}_F$ at $\mathfrak{P}$.
Then using similar arguments as in \cite[Proposition 8.2.3]{howardmadapusi}, we have the following proposition.

\begin{proposition}
There is an isomorphism \[\bigsqcup_{\Q^{\times}_{>0}\backslash \mathbb{A}_f^{\times}/K_0} \widehat{\mathcal{T}}_{\Xi}(\sigma)_{/R}\rightarrow \widehat{\mathcal{M}}_{\Xi,0}(\sigma)/R \]
of formal Deligne--Mumford stacks over $R$ whose base change to $\C$ agrees with \Cref{p:local-chart}. 
Moreover, the map 
\[\widehat{\mathcal{M}}_{\Xi,0}(\sigma)/R\rightarrow \widehat{\mathcal{M}}_{\Xi}(\sigma)/R\]
is an \'etale map of Deligne--Mumford stacks given as the quotient by $K_\Xi/K_{\Xi,0}$.
\end{proposition}
The proof follows from the description given over $\C$ \ref{p:local-chart}, the flatness of both sides over $\Z_{(p)}$ and the fact the normalization of $\mathrm{Spec}(\Z_{(p)})$ in  $\mathrm{Sh}_{K_0}(\mathbb{G}_m,\mathcal{H}_0)$ is isomorphic to $\bigsqcup_{a\in\Q^{\times}_{>0}\backslash \mathbb{A}_f^{\times}/K_0}\mathrm{Spec}(R)$, see \cite[Proposition 8.2.3]{howardmadapusi} for a proof and more details. 
\subsection{Special divisors}\label{s:special-div-main}
We continue to assume in this section that the lattice $(L,Q)$ is maximal and let  $\Sigma$ be a smooth $K$-admissible cone decomposition. 

For every $\beta\in L^\vee/L$, $m\in Q(\beta)+\Z$ such that $m>0$, one can define a {\it  special divisor} $\mathcal{Z}(\beta,m)\rightarrow \mathcal{M}$ following \cite[Definition 4.5.6]{agmp-annals}. We recall briefly the definition and refer to {\it loc. cit.} for more details.

The Shimura variety $\mathcal{M}$ carries the family of Kuga-Satake abelian varieties $\mathcal{A}\rightarrow \mathcal{M}$. For any scheme $S\rightarrow \mathcal{M}$, a group of special quasi-endomorphisms $V_\beta(\mathcal{A}_S)$ is defined in \cite[Section 4.5]{agmp-annals}. Then the functor sending a scheme $S$ to \[\mathcal{Z}(\beta,m)(S)=\{x\in V_\beta(\mathcal{A}_S)| Q(x)=m\}\]
is representable by a Deligne--Mumford stack which is \'etale locally an effective Cartier divisor on $\mathcal{M}$. We will rather consider its image in $\mathcal{M}$ by a procedure described in \cite{howardmadapusi} after Proposition 6.5.2. By abuse of notation, we also denote by $\mathcal{Z}(\beta,m)$ its closure in  $\mathcal{M}^{\Sigma}$, which is again a Cartier divisor. 
\medskip 

In what follows, we will give an explicit description of $\mathcal{Z}(\beta,m)$ in the formal completions of $\mathcal{M}^{\Sigma}$ along its boundary components. Since for our purposes we only need $\beta=0$ and $m$ coprime to $p$, we will only describe what happens in this situation and we abbreviate for short $\mathcal{Z}(\beta,m)=\mathcal{Z}(m)$. We assume that $m\geq 1$ is coprime to $p$ for the rest of this section. 
\medskip

By \cite[Page 434]{agmp-annals}, $\cZ(m)(\C)$ has a complex uniformization as follows: for any $g\in G(\mathbb{A}_f)$, let $L_g=g.\widehat{L}\cap L_\Q$ and consider the sub-Hermtian domain of $\D$
\[\D^{\circ}(\lambda)=\{x\in\D^{\circ}| (x,\lambda)=0\},\]
 where $\lambda\in L_g$, $Q(\lambda)=m$. Then $\cZ(m)(\C)$ is equal to the union of $\D^{\circ}(\lambda)$ for $g\in G(\mathbb{A}_f)$ and $\lambda\in L_g$ with $Q(\lambda)=m$. 
 
For any $\lambda \in L_g$ with $Q(\lambda)=m$, let $G_\lambda$ be the fixator of $\lambda$, $L_\lambda$ the orthogonal lattice to $\lambda$ in $L_\Q$, and let $\D_\lambda\subset \D$ be the orthogonal to $\lambda$. Notice that $\D_\lambda$ does not depend on $g$ but only on $\lambda\in L_\C$. Notice that since $m$ is coprime to $p$, the lattice $L_\lambda$ is also maximal at $p$. Then $(G_\lambda,D_\lambda)$ is again a Shimura datum of GSpin associated to the lattice $(L_\lambda,Q)$ which is of signature $(b-1,2)$ and has reflex field equal to $\Q$. 
If we choose $K_\lambda\subset G_\lambda(\mathbb{A}_f)$ a compact open subgroup as in \cite[Equation (4.1.2)]{agmp-annals}, then $K_\lambda \subset K\cap G_\lambda(\mathbb A_f)$ and we obtain a morphism of complex Shimura varieties
\[M_\lambda(\C)\rightarrow M(\C).\]
By the description \cite[Equation (2.4)]{agmp-annals}, the union over $g\in G(\mathbb{A}_f)$, $\lambda\in L_g$ with $Q(\lambda)=m$ of the images of $M_\lambda(\C)$ is equal to $\cZ(m)(\C)$.

Now since $(G_\lambda,D_\lambda)$ is again a Shimura variety of GSpin type associated to a lattice maximal at $p$, the discussion in the previous sections applies verbatim to the Shimura variety $M_\lambda$ and yields similar description for the compactification and the integral model over $\Z_{(p)}$. In particular, we have a map between integral models $\cM_\lambda\rightarrow \cM$ over $\Z_{(p)}$ which factors through $\cZ(m)$ by \cite[Page 82]{howardmadapusi}.
\[\cM_\lambda\rightarrow \cZ(m)\hookrightarrow \cM\]
and the union over of images\footnote{This union is in fact finite.} of such maps for $g\in G(\mathbb{A}_f)$ and $\lambda\in L_g$ with $Q(\lambda)=m$ is equal to $\cZ(m)$.

Let $(\Phi,\sigma)$ be a toroidal stratum representative for $\cM$. From the description of the parabolic subgroups of $\mathrm{GSpin}(b,2)$, we have the following lemma.
\begin{lemme}\label{admissibi}
The group $P\cap G_\lambda$ is an admissible parabolic subgroup of $G_\lambda$ if and only if $\lambda\in I_\Phi^{\bot}$.
\end{lemme}
Notice also that if $\lambda\notin I_\Phi^{\bot}$, then the image of $\D_\lambda$ in $ M^{\Sigma}(\C)$ will not intersect the boundary components parameterized by $\Phi$, as its projection to the Baily-Borel compactification will not do so. Hence they will not appear in the formal completions of $\cM^{\Sigma}$ along these boundary components.

We can write $\Phi=(P,\D^{\circ},h)$ and let $\lambda \in L_g$ with $Q(\lambda)=m$ such that $\lambda\in I_\Phi^{\bot}$. \Cref{admissibi} shows that $(\Phi,\sigma)$ can also be seen as a toroidal stratum representative with respect to $(G_\lambda,\D_\lambda)$ by considering $P\cap G_\lambda$, see \cite[Section 2.1.28]{madapusiintegral} for more details. Let $\mathcal{M}_{\lambda,\Upsilon}$ be the integral model over $\Z_{(p)}$ of the mixed Shimura variety associated to $\Phi$. We get then a morphism of mixed Shimura varieties 
\[\mathcal{M}_{\lambda,\Phi}\rightarrow \mathcal{M}_\Phi,\]
as well as a morphism of partial compactifications respecting the strata
\[\mathcal{M}_{\lambda,\Phi}(\sigma)\rightarrow \mathcal{M}_\Phi(\sigma).\] 
By \cite[Proposition 2.1.29]{madapusitor}, the morphism induced at the level of formal completions along the boundary strata given by $\sigma$ is compatible with the toroidal compacitifications of $\mathcal{M}_\lambda$ and $\mathcal{M}$. In particular, we get a commutative diagram
\begin{align*}
\xymatrix{\widehat{\mathcal{M}}_{\lambda,\Phi}(\sigma) \ar[rr]\ar[d] && \widehat{\mathcal{M}}_{\Phi}(\sigma)\ar[d]\\ 
\widehat{\mathcal{Z}}(m) \ar[rr] && \widehat{\mathcal{M}}^{\Sigma}, 
}
\end{align*}
where the right vertical map is an \'etale cover of Deligne--Mumford stacks, the left vertical map is an \'etale cover of an open and closed subset by \cite[Page 82]{howardmadapusi}. Finally, the union over $g \in G(\mathbb{A}_f)$, $\lambda\in L_g$ with $Q(\lambda)=m$ of the images of the left map covers the whole $\widehat{\cZ}(m)$.
\subsubsection{Special divisors along type II boundary components}\label{ss:special-tII}
 Let $(\Upsilon,\sigma)$ be a toroidal stratum representative of type II. 

Let $\lambda \in L$ with $Q(\lambda)=m$ such that $\lambda\in I_\Upsilon^{\bot}$ and $m$ is coprime to $p$. We have a morphism of formal completions of the partial compactifications of mixed Shimura varieties 
\[\widehat{\mathcal{M}}_{\lambda,\Upsilon}(\sigma) \rightarrow \widehat{\mathcal{M}}_{\Upsilon}(\sigma).\]
Let $x\in \mathcal{B}^{\Upsilon,\sigma}(\overline{\F}_p)\subset \mathcal{M}_\Upsilon(\sigma)(\overline{\F}_p)$ and let $\mathcal{O}_{\mathcal{M}_{\Upsilon}(\sigma),x}$ be the local ring at $x$. Let $\overline{x}$ be the image of $x$ in $\overline{\mathcal{M}}_\Upsilon(\overline{\F}_p)$ and let $z$ the image in $\mathcal{M}^h_\Upsilon(\overline{\F}_p)$.
If follows from \Cref{diagram-typeII} that the formal completion $\widehat{\mathcal{O}}_{\mathcal{M}_\Upsilon(\sigma),x}$ is isomorphic to \[\widehat{\mathcal{O}}_{\mathcal{M}_\Upsilon(\sigma),x}\simeq \Z_p[[X]]\widehat{\otimes} \widehat{\mathcal{O}}_{\overline{\mathcal{M}}_{\Phi},\overline{x}}.\]
Moreover, the pull-back of the torsor $\overline{\mathcal{M}}_\Upsilon\rightarrow \mathcal{M}^h_\Upsilon$ to $\mathrm{Spf}(\widehat{\mathcal{O}}_{\mathcal{M}^h_\Upsilon,z})$ is trivial, as it is trivial by reduction to $\overline{\mathbb{F}}_p$ and we can lift formally any  section. Hence
\[\widehat{\mathcal{O}}_{\overline{\mathcal{M}}_{\Phi},\overline{x}}\simeq \widehat{\mathcal{O}}_{D\otimes\mathcal{E},\overline{x}}\]

For $\lambda\in D$, consider the map over $\mathcal{M}^h_\Upsilon$ 
\begin{align}\label{kernel}
    D\otimes \mathcal{E}\xrightarrow{(\lambda.\,)\otimes \mathrm{Id}} \mathcal{E}.
\end{align} 
Its kernel is flat over $\mathcal{M}^h_\Upsilon$. Let $I_\lambda\subset\lambda$ be the ideal defining it. Then $\widehat{I}_\lambda\hookrightarrow \widehat{\mathcal{O}}_{D\otimes\mathcal{E},\overline{x}}$ is flat over $\widehat{\mathcal{O}}_{\mathcal{M}^h_\Upsilon,z}$.

\begin{proposition}\label{ss:dia-sp-II}
The formal completion $\widehat{\cZ}(m)$ along $x$ is the union over $\lambda\in D$ with $Q(\lambda)=m$ of the vanishing loci inside $\widehat{ \mathcal{O}}_{\mathcal{M}_{\Upsilon}(\sigma),x}$ of the ideals $\Z_p[[X]]\widehat{\otimes} \widehat{I}_\lambda$. 
\end{proposition}
\begin{proof}
Let $\lambda\in L$ such that $\lambda\in J^{\bot}$ and $Q(\lambda)=m$. Then we have a description of the mixed Shimura variety $M_{\Upsilon,\lambda}$ similar to \Cref{diagram-typeII}, namely, it has a fibration structure which fits into the following diagram: 

\begin{align*}
\xymatrix{\widehat{\mathcal{M}}_{\lambda,\Upsilon}(\sigma) \ar[rr]^{\widehat{\mathcal{T}}_{\Upsilon}(\sigma)} \ar[d] && \overline{\mathcal{M}}_{\lambda,\Upsilon}\ar[rr]^{D_\lambda\otimes\mathcal{E}} \ar[d] && \mathcal{M}^h_{\lambda,\Upsilon} \ar[d]
 \\ 
\widehat{\mathcal{M}}_\Upsilon(\sigma) \ar[rr]^{\widehat{\mathcal{T}}_\Upsilon(\sigma)} && \overline{\mathcal{M}}_\Upsilon\ar[rr]^{D\otimes\mathcal{E}} && \mathcal{M}^h_\Upsilon. 
}
\end{align*}
One can check that $D_\lambda=\overline{\lambda}^{\bot}$ where $\overline{\lambda}$ is the image of $\lambda$ in $D=J^{\bot}/J$. 
Moreover, the right vertical map in the above diagram is an \'etale cover and the vertical middle map is equivariant with respect to the inclusion \[\overline{\lambda}^{\bot}\otimes\mathcal{E}\hookrightarrow D\otimes\mathcal{E},\]
and the left vertical map has image given by an open and closed subset of $\widehat{\cZ}(m)$.

Let $z'\in \mathcal{M}^h_{\lambda,\Upsilon}(\overline{\F}_p)$ be a point mapping to $z$, then  $\widehat{\mathcal{O}}_{\mathcal{M}^h_{\lambda,\Upsilon,z'}}\simeq\widehat{\mathcal{O}}_{\mathcal{M}^h_{\Upsilon,z}}$. 
Hence the above diagram becomes at the level of completed local rings:
\begin{align*}
\xymatrix{\mathrm{Spf}(\Z_p[[X]]\widehat{\otimes}\widehat{\mathcal{O}}_{\overline{\lambda}^\bot\otimes\mathcal{E},\overline{x'}}) \ar[rr] \ar[d] && \mathrm{Spf}(\widehat{\mathcal{O}}_{\overline{\lambda}^\bot\otimes\mathcal{E},\overline{x'}})\ar[rr] \ar[d] && \mathrm{Spf}(\widehat{\mathcal{O}}_{\mathcal{M}^h_{\lambda,\Upsilon,z'}}) \ar[d]^{\simeq}
 \\ 
\mathrm{Spf}(\Z_p[[X]]\widehat{\otimes} \widehat{\mathcal{O}}_{D\otimes\mathcal{E},\overline{x}}) \ar[rr] &&\mathrm{Spf}(\widehat{\mathcal{O}}_{D\otimes\mathcal{E},\overline{x}}) \ar[rr] && \mathrm{Spf}(\widehat{\mathcal{O}}_{\mathcal{M}^h_{\Upsilon,z}}). 
}
\end{align*}

where the vertical map is contained in the kernel of the map \ref{kernel}. By considering all the $\lambda\in J^{\bot}$ that map to a given class $\overline{\lambda}\in D$, we get that the image is exactly the kernel of the map \ref{kernel} and hence the image of left vertical map is defined by the ideal $\Z_p[[X]]\widehat{\otimes} \widehat{I}_\lambda$, see \cite[Equation (26)]{zemel} for a description over $\C$. 
Finally, since $\widehat{\cZ}(m)$ is equal to the union of such images, the conclusion follows.
\end{proof}
\subsubsection{Special divisors along type III boundary components}\label{ss:special-tIII}
Let $(\Xi,\sigma)$ be a stratum representative of type III. Let $K_I=I^{\bot}/I$ be the Lorentzian lattice as introduced in \Cref{s:typeIII} and we continue to assume that $\sigma$ is a one dimensional inner ray. Let $\omega\in K_I\cap C_\Xi$ be a generator of $\sigma$ with $(\omega.\omega)=-2N$, $N\geq 1$. Let $\omega'\in K_I^{\vee}$ be an element such that $(\omega.\omega')=1$. 


Let $\lambda\in L$ with $Q(\lambda)=m$ and such that $\lambda \in I^{\bot}$. The projection $\overline{\lambda}\in K_I$ defines a divisor in the torus $\mathcal{T}_\Xi= \mathrm{Spec}(\Z_{(p)}[q^\alpha]_{\alpha\in {\Gamma_\Xi^{\vee}}})$ given by the equation $q^{\overline{\lambda}}=1$. 

In the partial compactification $\mathcal{T}_\Xi\hookrightarrow \mathcal{T}_\Xi(\sigma)$, the equation of this divisor becomes $q^{\overline{\lambda}}-1=0$ if $(\overline \lambda.\omega)\geq 0$ or $q^{-\overline \lambda}-1=0$ otherwise. Notice also that this divisor intersects the toric boundary divisor defined by $\sigma$ if and only if $(\omega.\overline \lambda)=0$. We will hence restrict ourselves to this latter situation and we denote by \[\mathcal{T}_{\Xi}(\lambda,\sigma)\hookrightarrow\mathcal{T}_{\Xi}(\sigma) \] the divisor defined by $\lambda$. By construction, it only depends on the class of $\lambda$ in $K_I$. 
\begin{proposition}\label{equation-typeIII}
Let $\widehat{\cZ}(m)$ be the formal completion of $\cZ(m)$ along the boundary component of $\mathcal{M}^{\Xi}$ index by $(\Xi,\sigma)$.  
Then the following diagram is commutative and compatible with \Cref{p:local-chart}. 
\[
\xymatrix{
\underset{{ a\in  \Q^\times_{>0} \backslash \mathbb{A}_f^\times / K_0  }}{\bigsqcup}\underset{{\underset{Q(\lambda)=m, (\lambda.\omega)=0 }{\lambda \in K}}}{\bigsqcup} \widehat{\mathcal{T}}_{\Xi,0}(\lambda,\sigma)_{/\C}  \ar[rr] \ar[d] && \underset{{ a\in  \Q^\times_{>0} \backslash \mathbb{A}_f^\times / K_0  }}{\bigsqcup}\widehat{\mathcal{T}}_{\Xi,0}(\sigma)
  \ar[d]
 \\ 
\widehat{\mathcal{Z}}(m) \ar[rr] & & \widehat{\mathcal{M}}^{\Sigma} 
}
\]
The vertical maps are \'etale coverings of formal Deligne--Mumford stacks and the union over $\lambda\in I^{\bot}$ covers $\widehat{\mathcal{Z}}(m)$.
\end{proposition}
\begin{proof}
Let $\lambda \in L\cap I^{\bot}$ with $Q(\lambda)=m$ and such that projection $\overline{\lambda}\in I^{\bot}/I$ is orthogonal to $\omega$. Then we have similarly a description of the mixed Shimura variety $\mathcal{M}_{\lambda,\Xi}$ associated to the Shimura datum $(G_\lambda,\D_\lambda)$ as a torus fibration and such that the following diagram is commutative

\begin{align*}
\xymatrix{\widehat{\mathcal{M}}_{\lambda,\Xi}(\sigma) \ar[rr]^{\widehat{\mathcal{T}}_{\Xi,0}(\lambda,\sigma)} \ar[d] && \mathcal{S}(\mathbb{G}_m,\mathcal{H}_0)_{/R}\ar[d]\\ 
\widehat{\mathcal{M}}_\Xi(\sigma) \ar[rr]^{\widehat{\mathcal{T}}_{\Xi,0}(\sigma)}&&\mathcal{S}(\mathbb{G}_m,\mathcal{H}_0)_{/R}.
}
\end{align*}
The left vertical map is equivariant with respect to the inclusion $\widehat{\mathcal{T}}_{\Xi}(\lambda,\sigma)\hookrightarrow \widehat{\mathcal{T}}_{\Xi}(\sigma)$ and its image only depends on $\overline{\lambda}\in I^{\bot}/I$. Since the formal completion $\widehat{\cZ}(m)$ is the union over $\lambda\in L$ of the images of the left vertical maps, we get the desired result.
\end{proof}

\section{Arithmetic intersection theory and modularity}

We recall in this section the Arakelov arithmetic intersection theory on $\mathcal{M}^{\Sigma}$ following \cite{burgos},  the modularity results of the special divisors from \cite{howardmadapusi, borcherdszagier} and its extension to complex toroidal compactification by \cite{bruinierzemel}. Then we derive a further extension to the integral model of the toroidal compactifications of GSpin Shimura varieties. 
\subsection{Modularity of special divisors} 
Let $(L,Q)$ be a maximal quadratic lattice with signature $(b,3)$ and assume that $b\geq 3$. 

Let $K\subset G(\mathbb{A}_f)$ be the compact open subgroup from \Cref{ss:gspinshimura} and let $\Sigma$ be a $K$-admissible smooth polyhedral cone decomposition. Denote by $\mathcal{M}^{\Sigma}$ the toroidal compactification of the integral model of the GSpin Shimura variety constructed in \Cref{s:integral-models}. Let $\widehat{\CH}^1(\cM^{\Sigma},\mathcal{D}_{pre})_{\Q}$ be the first Chow group of prelog forms as defined in \cite[Definition 1.15]{burgos}.
\medskip

Let $\Upsilon$ be a cusp label representative of type II. Then there is a unique one dimensional ray in the cone decomposition associated to $\Upsilon$ and we denote by abuse of notation $\mathcal{B}^\Upsilon$ the closure of the boundary divisor associated to $\Upsilon$. 

Consider now $(\Xi,\sigma)$ a toroidal stratum representative of type III such that  $\sigma$ is a $1$-dimensional inner ray in the cone decomposition $\Sigma$. Then we denote by  $\mathcal{B}
^{\Xi,\sigma}$ the closed boundary divisor in $\mathcal{M}^{\Sigma}$ associated to $(\Xi,\sigma)$.

Let $\beta\in L^{\vee}/L$ and $m\in Q(\beta)+\Z$ with $m>0$. For every toroidal stratum representative $\Upsilon$ and $(\Xi,\sigma)$, let $\mu_{\Upsilon}(\beta,m)$ and $\mu_{\Xi,\omega}(\beta,m)$ be the real numbers defined by \Cref{Eq:mult-typeII} and \Cref{Eq:mult-typeIII}, see also \cite{bruinierzemel}. Consider then the following divisor on $\mathcal{M}^{\Sigma}$: 

\begin{align}\label{eq:special-divisor}
    \mathcal{Z}^{tor}(\beta,m)=\mathcal{Z}(\beta,m)+\sum_{\Upsilon }\mu_{\Upsilon}(\mu,m)\cdot \cB^{\Upsilon}+\sum_{(\Xi,\omega)}\mu_{\Xi,\omega}(\mu,m)\cdot \cB^{\Xi,\omega},
\end{align}
where the two last sums are over toroidal stratum representatives of type II and type III respectively. 
Then by \cite{bruinierzemel}, the Cartier divisor $\mathcal{Z}^{tor}(\beta,m)$ can be endowed with a Green function $\Phi_{\beta,m}$ such that the resulting pair 
\[\widehat{\cZ}^{tor}(\beta,m)=(\cZ^{tor}(\beta,m),\Phi_{\beta,m})\]
is an element of the first Chow group of prelog forms $\widehat{\CH}^1(\cM^{\Sigma},\mathcal{D}_{pre})_{\Q}$.
For $m=0$ and $\beta=0$, we define $\widehat{\cZ}(0,0)$ to be any arithmetic divisor whose is class is the dual of the hermitian line bundle $\widehat{\mathcal{L}}=(\mathcal{L},||.||_{pet})$ endowed with the Petersson metric $||z||^2=[z,\overline{z}]$.

\medskip

Consider then the following generating series \[\Phi_L:=\sum_{\beta\in L^\vee/L}\sum_{m\in Q(\beta)+\Z}\widehat{\mathcal{Z}}^{tor}(\beta,m)q^m e_\beta\in\C[L^\vee/L][[q^{\frac{1}{D_L}}]]\otimes\widehat{\CH}^1(\cM^{\Sigma},\mathcal{D}_{pre})_{\Q},\]

where $(e_\beta)_{\beta\in L^\vee/L}$ is a basis of the $\C$-vector space $\C[L^\vee/L]$, $D_L$ is the discriminant of $L$, and $q=e^{2i\pi\tau}$, where $\tau\in\mathbb{H}$ is in the the upper-half plane.
\medskip

Let \[\rho_L:\mathrm{Mp}_2(\Z)\rightarrow \mathrm{Aut}_\C(\C[L^\vee/L])\] be the Weil representation associated to the quadratic lattice $(L,Q)$, where $\mathrm{Mp}_2(\R)$ is the metaplectic double cover if $\mathrm{Mp}_2(\R)$. For $k\in \frac{1}{2}\Z$, let $\mathrm{Mod}_{k}(\rho_L)$ denote the vector space of vector valued modular forms of weight $k$ with respect to $\rho_L$. We then have the following theorem.

\begin{theorem}\label{Theorem-modularity}
The generating series $\Phi_L$ is the Fourier development of a vector-valued modular forms of weight $1+\frac{b}{2}$ and representation $\rho_L$, i.e., \[\Phi_L\in \mathrm{Mod}_{1+\frac{b}{2}}(\rho_L)\otimes \widehat{\CH}^1(\cM^{\Sigma},\mathcal{D}_{pre})_{\Q}.\]
\end{theorem}
\begin{proof}
Let $F\in M^{!}_{1-\frac{b}{2}}(\overline{\rho_L})$ be a weakly holomorphic modular form of weight $1-\frac{b}{2}$ with respect to the complex conjugate Weil representation of $\rho_L$ such that that $F$ has integral principal part, and let $\Psi$ be the associated Borcherds product. Then by \cite[Theorem 5.5]{bruinierzemel}, the divisor in $ \cM^{\Sigma}(\C)$ of $\Psi(F)_\C$ is equal to 
\[\sum_{\beta\in L^{\vee}/L}\sum_{m\in Q(\beta)+\Z}c_\beta(-m)\cZ^{tor}(\beta,m)(\C).\]

Since Borcherds products are defined rationally by \cite[Theorem A]{howardmadapusi}, we only need to check that the divisor of the Borcherds products has the expected form over $\Z$ and this will be automatic if all the special divisors and the boundary divisors are flat. By \cite[Theorem 4.1.5]{madapusitor}, the boundary divisors are flat and by \cite[Proposition 7.2.2]{howardmadapusi}, the special divisors are flat over $\Z[\frac12]$ and over $\Z$ if $b\geq 4$. For $b=3$, one can use the algebraic version of the Borcherds embedding trick as in \cite[Section 9.2]{howardmadapusi} to prove that no further components appear at $2$ and hence the divisor of the Borcherds product has the correct form. Hence we conclude by the criterion in \cite[Proposition 5.4]{bruinierzemel}.
\end{proof}


\section{The main estimates and proof of the main theorems}

We state in this section the local and global estimates that will allow us to prove \Cref{th:K3nf} and \Cref{th:K3ff}. Then we will prove the global estimates and we postpone the proof of local estimates to the next section.

\subsection{Number field setting}

Let $X$ be K3 surface over a number field $K$. Given an embedding $\tau:K\inj \C$, let $(L,Q)$ be a maximal lattice containing the transcendental lattice of $X^\tau(\C)$. It is an even lattice of signature $(b,2)$ whose genus is independent from $\tau$. We can assume furthermore that $b\geq 3$, as the case $b\leq 2$ has already been treated, see \cite{charles-exceptional-isogenies,shankar-tang}. 

Let $\mathcal{M}$ be the integral model of the GSpin Shimura variety associated to the lattice $(L,Q)$ and, given an admissible polyhedral cone decomposition $\Sigma$, let $\mathcal{M}^{\Sigma}$ be its toroidal compactitfication as in Section 2. 
By \cite{madapusiperatate}, the K3 surface has an associated Kuga-Satake abelian variety which we can also assume to be defined over the number field $K$, up to taking a finite extension. Hence it defines a $K$-point of $\mathcal{M}^{\Sigma}$. 
By the extension property of the integral model, there exists $N\geq 1$ such that, up to taking a finite extension of $K$, we have a flat  morphism over $\Z$:
\[\mathrm{Spec}(\mathcal{O}_K\left[\frac1N\right])\rightarrow \mathcal{M},\]
and by properness, this map extends to
\[\rho:\cY=\mathrm{Spec}(\mathcal{O}_K)\rightarrow \mathcal{M}^{\Sigma}.\]
By construction, the image of this map is not contained in any special divisor. A prime over $N$ is said to be a prime of bad reduction and otherwise of good reduction.

As in \cite[Theorem 2.4]{sstt}, we will rather prove the following more general version, which is easily seen to imply \Cref{th:K3nf}.
\begin{theorem}\label{t:main_sp_end}
Let $\cY\in \mathcal{M}^{\Sigma}(\mathcal{O}_K)$ with smooth reduction outside $N$. Let $D\in \Z_{>0}$ be a fixed integer represented by $(L,Q)$ and coprime to $N$. Assume that $\cY_K\in M(K)$ is not contained in any special divisor $\cZ(m)(K)$. Then there are infinitely many places $\fP$ of $K$ of good reduction such that $\cY_{\overline{\fP}}$ lies in the image of $\cZ(Dm^2)\rightarrow \cM$ for some $m\in \Z_{>0}$ coprime to $N$.
\end{theorem}

Let $\rho:\cY\rightarrow\mathcal{M}^{\Sigma}$ be as in the previous theorem. We first begin by the following proposition. 
\begin{proposition}\label{choice-toroidal-nf}
There exists a refinement of the cone decomposition $\Sigma$, such that the map $\rho:\cY\rightarrow \cM^{\Sigma}$ satisfies the following property: 
for any prime $\mathfrak{P}$ of bad reduction, the image of the closed point $\{\mathfrak{P}\}$ under $\rho$ is contained in a stratum which is a locally closed divisor of $\mathcal{M}^{\Sigma}$. 
\end{proposition}
\begin{proof}
Let $s_\mathfrak{P}\in \cY$ be the closed point $\mathfrak{P}$ of $\cY$. By \Cref{eq:stratification-integral}, the image of $s_\mathfrak{P}$ lies in a stratum indexed either by either a type II boundary component $\Upsilon$ or a type III $(\Xi,\sigma)$ toroidal stratum representative. In the type II case, the boundary is already a divisor and there is nothing to prove. In the type III case, let $r$ be the dimension of the cone $\sigma$. Then we get a morphism: 
\begin{align}\label{map-local}
    \mathrm{Spf}(W(\overline{\mathbb{F}}_p))\rightarrow \widehat{\mathcal{M}}^{\Sigma},
\end{align}
where $\widehat{\mathcal{M}}^{\Sigma}$ is the formal completion along the boundary component defined by $(\Xi,\sigma)$. By a similar analysis to \Cref{s:typeIII-integral}, we have an \'etale cover of formal Deligne--Mumford stacks: 
\[\widehat{\mathcal{T}}_\Xi(\sigma)\rightarrow \widehat{\mathcal{M}}^{\Sigma}.\]
Hence the map \Cref{map-local} lifts to a morphism
\begin{align}\label{eq:tangent-map}
    \mathrm{Spf}(W(\overline{\mathbb{F}}_p))\rightarrow \widehat{\mathcal{T}}_\Xi(\sigma),
\end{align}
where \[\widehat{\mathcal{T}}_\Xi(\sigma)=\mathrm{Spf}\left(\Z_p[q^{\alpha}|(\alpha,\sigma)=0]\otimes_{\Z_p}\Z_p[[q^{\alpha}|(\alpha,\sigma)>0]] \right).\] Hence this corresponds to a map 
\[\Z_p[[q^{\alpha}|(\alpha,C)>0]]\otimes \Z_p[q^{\alpha}|(\alpha,C)=0]\rightarrow W(\overline{\mathbb{F}}_p).\] 

The linear form on $\Gamma_\Xi^{\vee}$ given by sending an element $\alpha$ to the $p$-adic valuation of the image of $q^{\alpha}$ under the above map is represented by an element $\omega\in \Gamma_\Xi$ which satisfies $(\omega.\alpha)>0$ whenever $(\alpha.\sigma)>0$, hence $\omega$ is in $\sigma$. The cocharacter defined by $\omega$ is in fact tangent to the map given in \Cref{eq:tangent-map}. Let $\sigma'$ in $\sigma$ be the ray defined by $\omega$ and let $\Sigma'$ be the new cone decomposition obtained by refining $\Sigma$ and which contains $\sigma'$ as a one dimensional ray. Then $\mathcal{M}^{\Sigma'}$ is a blow-up of $\mathcal{M}^{\Sigma}$ and by the preceding discussion, the point $s_\mathfrak{P}$ belongs to the boundary divisor parameterized by $(\Xi,\sigma')$. Since there are only finitely many primes of bad reduction, then by repeating this procedure finitely many times, we get the desired cone decomposition. 
\end{proof}
\medskip

We will work from now on with the toroidal compactification given by the above proposition.  For $m\geq 1$ an integer, let $\cZ(m)$ be the closed special divisor $\cZ(0,m)\hookrightarrow \cM^{\Sigma}$ and $\widehat{\mathcal{Z}}^{tor}(m)$ the arithmetic divisor associated to $\mathcal{Z}^{tor}(m)$ by \Cref{eq:special-divisor}. 
The pullback via the period map $\rho:\mathcal{Y}\rightarrow \mathcal{M}^{\Sigma}$ allows us to define the height $h_{\widehat{\mathcal{Z}}^{tor}(m)}(\cY)$ of $\cY$ with respect to the arithmetic divisor $\widehat{\mathcal{Z}}^{tor}(m)$ as its image under the composition: 
\begin{align*}
    \widehat{\CH}^1(\cM^{\Sigma},\mathcal{D}_{pre})_{\Q}\rightarrow &\widehat{\CH}^1(\cY,\mathcal{D}_{pre})_{\Q}\xrightarrow{\widehat{\deg}} \R,\\
    \widehat{\mathcal{Z}}^{tor}(m)&\longrightarrow h_{\widehat{\mathcal{Z}}^{tor}(m)}(\cY).
\end{align*}
By choice of the lattice $(L,Q)$, the arithmetic curve $\cY$ intersects properly the divisors $\cZ(m)$, $\cB^{\Xi,\omega}$ and $\cB^\Upsilon$ for every $\Upsilon$ and $(\Xi,\omega)$. Hence we have 
\begin{align}\label{intersectionformula2}
h_{\widehat{\mathcal{Z}}^{tor}(m)}(\mathcal{Y})=\sum_{\tau:K\hookrightarrow\C}\Phi_{m}(\cY^\tau)+\sum_{\mathfrak{P}}(\cY.\cZ^{tor}(m))_{\fP}\log|\mathcal{O}_{K}/\mathfrak{P}|,
\end{align}
where for $\tau:K\hookrightarrow \C$, we use $\cY^\tau$ to denote the point in $M(\C)$ induced by \[\Spec(\C)\xrightarrow{\tau} \Spec({\cO_K})=\cY\rightarrow \cM^{\Sigma}.\] 

We have 
\begin{align}\label{Eq:decomposition}
(\cY.\cZ^{tor}(m))_{\fP}=(\cY.\cZ(m))_{\fP}+\sum_{\Upsilon}\mu_{\Upsilon}(m)(\cY.\cB^\Upsilon)_{\fP}+\sum_{(\Xi,\omega)}\mu_{\Xi,\omega}(m)\cdot (\cY.\cB^{\Xi,\omega})_{\fP}.  
\end{align}

Let us denote by $\cO_{\cY\times_{\cM^\Sigma}\cZ(m),v}$ the \'etale local ring of $\cY\times_{\cM}\cZ(m)$ at $v$, then 
\begin{align}\label{int_formula_finite1}
(\cY.\cZ(m))_{\fP}=\sum_{v\in\cY\times_{\cM}\mathcal{Z}(m)(\overline{\mathbb{F}}_\mathfrak{P})}\length(\cO_{\cY\times_{\cM^\Sigma}\cZ(m),v}),
\end{align}
where $\bF_\fP$ denotes the residue field of $\fP$. 

Let \[(\cY.\cZ(m))=\sum_{\mathfrak{P}}(\cY.\cZ(m))_{\fP}\log|\mathcal{O}_{K}/\mathfrak{P}|.\]
The first new contribution of this paper is to prove the following estimate which results from  Borcherds modularity and ad hoc bounds on the multiplicities $\mu_{\Upsilon}(m)$ and $\mu_{\Xi,\omega}(m)$.
\begin{proposition}\label{p:globalnf}
As $m\rightarrow \infty$, we have \[(\cY.\cZ(m))+\sum_{\tau:K\hookrightarrow\C}\Phi_{m}(\cY^\tau)=O(m^\frac{b}{2}).\]
\end{proposition}
As a corollary, we get the following bound, which is referred to as the diophantine bound in \cite[Equation (5.2)]{sstt}.
\begin{corollaire}\label{Diopbound}
For any finite place $\mathfrak{P}$, we have the following bound.

\[(\cY.\cZ(m))_\fP=O(m^{\frac{b}{2}}\log m),\quad \Phi_m(\cY^\tau)=O(m^{\frac{b}{2}}\log m).\]
\end{corollaire}

For our next estimate, we recall the notion of asymptotic density from \cite{sstt}: for a subset $S\subset \Z_{>0}$, the \emph{logarithmic asymptotic density} of $S$ is defined to be \[\displaystyle \limsup_{X\rightarrow\infty}\frac{\log |S_X|}{\log X},\] where $S_X:=\{a\in S \mid X\leq a < 2X\}$.
\medskip

 Recall from Theorems 5.7 and 6.1 in \cite{sstt} that we have the following estimate: 
\begin{proposition}\label{p:archimidean-place}
There exists a subset $S_{bad}\subset \Z_{>0}$ of zero logarithmic asymptotic density such that 
\[\sum_{\tau:K\hookrightarrow \C}\Phi_m(\cY^{\tau})=c(m)\log(m)+o(m^{\frac{b}{2}}\log(m)),\]
where $-c(m)\asymp m^{\frac{b}{2}}$ and is defined in \cite[Section 3.3]{sstt}.
\end{proposition}
For a prime $\mathfrak{P}$ of good reduction, i.e., where the intersection of $\cY$ and $\cZ(m)$ above $\mathfrak{P}$ is supported in $\cM$, we have the following estimate which follows easily from \cite[Theorem 7.1]{sstt}.
\begin{proposition}
Let $\mathfrak{P}$ be a finite place of good reduction. Let $D\in \Z_{\geq 1}$ coprime to $N$. For $X\in \Z_{>0}$, let $S_{D,X}$ denote the set \[\{m\in \Z_{>0}\mid X \leq m<2X,\, \frac{m}{D}\in \Z \cap (\Q^\times)^2,\, (m,N)=1\}.\] Then we have
\[\sum_{m\in S_{D,X}}(\cY . \cZ(m))_\fP=o(X^{\frac{b+1}{2}}\log X).\]
\end{proposition}
Finally, for a prime $\mathfrak{P}$ of bad reduction, we prove the following proposition which is the second new contribution of this paper.

\begin{proposition}\label{p:finite-place}
Let $\mathfrak{P}$ a finite place of bad reduction. Let $D\in \Z_{\geq 1}$ coprime to $N$. For $X\in \Z_{>0}$, let $S_{D,X}$ be the set defined in the previous proposition. Then we have
\[\sum_{m\in S_{D,X}}(\cY . \cZ(m))_\fP=o(X^{\frac{b+1}{2}}\log X).\]
\end{proposition}
\subsection{Function field setting}

We assume in this section that the lattice $(L,Q)$ is self-dual at $p$. Then the Shimura variety $\mathcal{M}$ has smooth reduction at $p$ and we denote its reduction by $\mathcal{M}_{\mathbb{F}_p}$. Given an admissible cone decomposition $\Sigma$, we denote by $\mathcal{M}_{\mathbb{F}_p}^{\Sigma}$ the reduction of the toroidal compactification $\cM^{\Sigma}$. We first give a  new formulation of \Cref{th:K3ff}, see \Cref{th:main-ff-general},  then we will give the main estimates that will allow us to prove the latter.
\medskip

Let $\mathcal{X}\rightarrow \mathscr{S}$ be a generically ordinary non-isotrivial family of K3 surfaces over a smooth curve $\scrS$ over $\overline\F_p$. The quadratic lattice $(L,Q)$ in this case corresponds to a maximal quadratic lattice orthogonal to the generic geometric Picard group in the K3 lattice. Hence $(L,Q)$ has discriminant coprime to $p$ by assumption and we get a period map by \cite[section 4]{madapusiperatate} 
\[\rho:\scrS\rightarrow \mathcal{M}_{\F_p},\]
which is a finite map and the image of the generic point is in the ordinary locus. The locus in $\scrS$ where the Picard rank jumps corresponds then exactly to the union over $m\geq 1$ of the intersections $\scrS\cap\cZ(m)_{\F_p}$. Hence \Cref{th:K3ff} follows from the following theorem. 
\begin{theorem}\label{th:main-ff-general}
Let $\scrS\rightarrow \mathcal{M}_{\F_p}$ be a finite map with generically ordinary image and not contained in any special divisor. Then there exists infinitely many closed points in $\scrS$ in the union of $\cZ(m)_{\mathbb{F}_p}$ for integers $m$ coprime with $p$.
\end{theorem}

Let $\mathscr{S}$ be a smooth curve as in the theorem above. By properness, we can extend the map \[\rho:\overline\scrS\rightarrow \mathcal{M}_{\F_p}^{\Sigma},\]
where $\overline\scrS$ is the smooth compactification of $\scrS$. 
We have the following proposition whose proof is similar to \Cref{choice-toroidal-nf} and hence we omit it.
\begin{proposition}\label{p:choice-tor-ff}
There exists a refinement of the cone decomposition $\Sigma$ such that the image of $\overline{\scrS}$ in $\mathcal{M}_{\overline{\F}_p}$ intersects the boundary only in strata corresponding to locally closed divisors. 
\end{proposition}
Let $\Sigma$ be a polyhedral cone decomposition which satisfies the conditions of the previous proposition. By abuse of notations, if $D\subset \mathcal{M}^{\Sigma}_{\mathbb{F}_p}$ is a Cartier divisor, we write
\[(D.\overline\scrS)=\deg_{\overline\scrS}\rho^{*}D.\]

We have then the following global estimate.    
\begin{proposition}\label{p:globalff}
As $m\rightarrow  \infty$, we have 
\[(\cZ(m)_{\overline{\F}_p}.\overline{\scrS})=|c(m)|(\overline{\scrS}.\mathcal{L}_{\F_p})+o(m^{\frac{b}{2}}).\]
\end{proposition}
For any integer $m$, we can decompose: 
\[(\cZ(m)_{\overline{\F}_p}.\overline{\scrS})=\sum_{P\in \overline{\F}_p}m_P(\cZ(m)_{\F_p},\overline{\scrS}),\]
where $m_P(\cZ(m)_{\F_p},\overline{\scrS})$ is the multiplicity of intersection at $P$. Our next goal is to estimate in average these local multiplicities and we start by the good reduction case already treated in \cite[Prop. 7.11, Th. 7.18 ]{maulik-shankar-tang-K3}.


Let $S$ be as in \cite[Section 7.1]{maulik-shankar-tang-K3}, i.e., a set of integers of positive density such that every $m\in S$ is coprime to $p$ and is representable by the quadratic lattice $(L,Q)$.

For $P\in (\scrS\cap\mathcal{M})(\F_p)$, we define as in \cite[Definition 7.6]{maulik-shankar-tang-K3}
\[g_P(m)=\frac{h_p}{p-1}|c(m)|,\]
where $h_p$ is the order of vanishing of the Hasse invariant at $P$, see {\it loc. cit.}
The following proposition is the combination of Proposition 7.11 and Theorem 7.18 from \cite{maulik-shankar-tang-K3}.
\begin{proposition}\label{p:local-good-ff}
Let $P\in \scrS(\overline{\F}_p)$. Then:
\begin{enumerate}
    \item If $P$ is not supersingular then 
    \[\sum_{m\in S_X}m_P(\cZ(m)_{\F_p}.\scrS)=O(X^{\frac{b}{2}}\log X).\]
    \item There exists an absolute constant $0<\alpha<1$ such that for any supersingular point $P$ we have \[\sum_{m\in S_X}m_P(\cZ(m)_{\F_p}.\scrS)=\alpha\sum_{m\in S_X}g_p(m)+O(X^{\frac{b+1}{2}}).\]
\end{enumerate}
\end{proposition}

Our new contribution in this setting is the following theorem which gives an estimate on intersection multiplicities at points where $\overline{S}$ intersects the boundary of $\cM^{\Sigma}_p$.
\begin{proposition}\label{t:local-bad-ff}
Let $P\in \overline{\scrS}(\overline{\F}_p)$ a point mapping to the boundary of $\mathcal{M}^{\Sigma}_{\mathbb{F}_p}$. Then we have the following estimate:
\[\sum_{m\in S_X}m_P(\cZ(m)_{\F_p}.\overline{\scrS})=O(X^{\frac{b}{2}}\log X).\]
\end{proposition}

\subsection{Proof of the main theorems}
Assuming the estimates in the previous section we now indicate how to prove \Cref{th:K3nf} and \Cref{th:K3ff}.
\begin{proof}[Proof of \Cref{th:K3nf}]
It is enough to prove \Cref{t:main_sp_end} in a similar way to \cite[Section 8]{sstt}. For convenience of the reader, we will sketch the proof. Assume for the sake of contradiction that there are only finitely many primes of good reduction such that $\cY$ intersects a special divisor of the form $\cZ(Dm^2)$ where $Dm^2$ is coprime with $N$ and is represented by $(L,Q)$.
By \Cref{p:globalnf} and \Cref{p:archimidean-place}, there exists a subset $S_{bad}\subset \Z_{>0}$ of logarithmic asymptotic density zero such that:
\[(\cY.\cZ(m))=-c(m)\log(m)+o(m^{\frac{b}{2}}\log(m))\asymp m^{\frac{b}{2}}\log(m).\]
Let $S^{good}_{D,X}=\{m\in S_{D,X}, m\notin S_{bad},\, (m,N)=1\}$, then one can easily check that $|S^{good}_{D,X}|\asymp X^{\frac{1}{2}}$ and $c(m)\gg X^{\frac{b}{2}}\log X$ for $m\in S^{good}_{D,X}$. Hence we get 
\begin{align}\label{contradiction}
    \sum_{m\in S^{good}_{D,X}}(\cY.\cZ(m))\asymp X^{\frac{b+1}{2}}\log X.
\end{align}
On the other hand, by Proposition 4.5 and \Cref{p:finite-place}, we get by summing over the finitely many places where either $\cY$ intersects a $\cZ(Dm^2)$ or which are of bad reduction
\[\sum_{m\in S^{good}_{D,X}}(\cY.\cZ(m))=o(X^{\frac{b+1}{2}}\log X),\]
which contradicts \Cref{contradiction}. 
\end{proof}
\begin{proof}[Proof \Cref{th:K3ff}]
The proof is similar: assume that there are only finitely many points in the union $\left(\cup_{m,m\wedge p=1}\cZ(m)\cap \scrS\right)(\overline{\F}_p)$ and let $S$ be a set as in Section 4.2. Then by \Cref{p:globalff}, we have 
\[\sum_{m\in S_X}(\cZ(m)_{\mathbb{F}_p}.\overline{\scrS})=\sum_{m\in S_X}|c(m)|(\overline{\scrS}.\mathcal{L}_{\mathbb{F}_p})+o(X^{b/2+1}).\]
On the other hand, by \Cref{p:local-good-ff} and \Cref{t:local-bad-ff} we have
\begin{align*}
    \sum_{m\in S_X}(\cZ(m)_{\F_p}.\overline{\scrS})&=\sum_{m\in S_X}\sum_{P\in \left(\cup_{m,m\wedge p=1}\cZ(m)\cap \scrS\right)(\overline{\F}_p)}m_P(\cZ(m)_{\F_p}.\overline{\scrS})\\
    &=\alpha\sum_{m\in S_X}|g_P(m)|+O(X^{\frac{b+1}{2}}).\\
    \leq \alpha\sum_{m\in S_X}|c(m)|(\overline{\scrS}.\mathcal{L}_{\mathbb{F}_p})+O(X^{\frac{b+1}{2}}),
\end{align*}
where the last equality results from the fact that the Hasse invariant is a section of $\mathcal{L}_{\mathbb{F}_p}^{\otimes p-1}$. These two estimates contradict each other, hence the result. 
\end{proof}

\subsection{Global estimate}
We prove in this section simultaneously Propositions \ref{p:globalnf} and \ref{p:globalff}.

By \Cref{Theorem-modularity}, the following generating series 
\[\sum_{\beta\in L^\vee/L}\sum_{m\in Q(\beta)+\Z}h_{\widehat{\mathcal{Z}}^{tor}(\beta,m)}(\cY)q^m e_\beta,\] and 
\[\sum_{\beta\in L^\vee/L}\sum_{m\in Q(\beta)+\Z}(\mathcal{Z}^{tor}(\beta,m)_{\F_p}.\overline{\scrS})q^m e_\beta\]
are elements of $\mathrm{Mod}_{1+\frac{b}{2}}(\rho_L)$. Classical estimates on the growth of coefficients of modular forms imply that (see \cite[Example 2.3]{tayouequi} for more details): 
\[h_{\widehat{\mathcal{Z}}^{tor}(m)}(\cY)=O(m^{\frac{b}{2}})\] 

and \[ (\mathcal{Z}(m)^{tor}_{\F_p}.\overline{\scrS})=|c(m)|(\overline{\scrS}.\mathcal{L}_{\mathbb{F}_p})+o(m^{\frac{b}{2}}).\]

By \Cref{intersectionformula2} and  \Cref{Eq:decomposition}, we can write 
\begin{align}
(\cY.\cZ(m))+\sum_{\tau:K\hookrightarrow\C}\Phi_{m}(\cY^\tau)= h_{\widehat{\mathcal{Z}}^{tor}(m)}(\cY) -\sum_{\Upsilon}\mu_{\Upsilon}(m)(\cY.\cB^{\Upsilon})_{\fP}\log|\mathcal{O}_{K}/\mathfrak{P}|\\-\sum_{\Xi}\mu_{\Xi,\sigma}(m)\cdot (\cY.\cB^{\Xi,\sigma})_{\fP}  \log|\mathcal{O}_{K}/\mathfrak{P}|
\end{align}
and similarly, we can write: 
\begin{align}
(\overline{\scrS}.\cZ(m)_{\F_p})= (\mathcal{Z}^{tor}(m)_{\F_p}.\overline{\scrS}) -\sum_{\Upsilon}\mu_{\Upsilon}(m)(\overline{\scrS}.\cB^{\Upsilon,\F_p})\\-\sum_{\Xi,\sigma}\mu_{\Xi,\omega}(m)\cdot (\overline{\scrS}.\cB^{\Xi,\sigma}_{\F_p}).
\end{align}

Hence we only have to bound the growth of the multiplicities $\mu_\Upsilon(m)$ and $\mu_{\Xi,\omega}(m)$.\footnote{$\omega$ is the unique integral generator of $\sigma$.} This is given by the following lemma. 
\begin{proposition}\label{p:growth-multi}
As $m\rightarrow\infty$, we have the following estimates. 
\begin{enumerate}
    \item For any type II cusp label representative $\Upsilon$, we have \[\mu_\Upsilon(m)\ll_\epsilon m^{\frac{b}{2}-1+\epsilon}.\]
    \item For any type III toroidal stratum representative $(\Xi,\sigma)$ such that $\sigma$ is a ray, we have  \[\mu_{\Xi,\omega}(m)\ll_\epsilon m^{\frac{b-1}{2}+\epsilon}.\]
\end{enumerate}
\end{proposition}
This proposition will be proved in the following two sections. 

\subsection{Estimates on type II multiplicities}\label{s:esti-typeII}

The goal of this section is to prove the type II estimate in \Cref{p:growth-multi}. First we recall some notations associated to isotropic planes introduced in \cite[Section 3.2]{bruinierzemel}.

Let $\Upsilon=(P,\D^{\circ},h)$ be a cusp label representative corresponding to a boundary component of type II. Recall from \Cref{s:typeII} that $P$ is the stabilizer of an isotropic plane $J_\Q$ and $J=J_\Q\cap h.L$ is a primitive isotropic plane of $h.L\cap L_\Q$. 

To simplify the notations, assume that $h.L\cap L_\Q=L$, the reader may otherwise replace $L$ by $L_h=h.L\cap L_\Q$ in what follows. Define then:
\[J_{L^\vee}=J_\R\cap L^\vee,\quad J^{\bot}_L=J^\bot\cap L,\quad J^{\bot}_{L^\vee}=J^\bot\cap L^\vee,\quad \textrm{and}\quad D=J^{\bot}_L/J.\]

The lattice $D$ is positive definite lattice of rank $b-2$. Its dual lattice can be described as \[D^\vee=J^{\bot}_{L^\vee}/J_{L^\vee}.\]  and the discriminant lattice is given by: 
\[\Delta_D=D^\vee/D=J^{\bot}_{L^\vee}/(J^\bot_L+J_{L^{\vee}}=L_{J}^\vee/(L+J_{L^\vee}),\]
where $L^\vee_J$ is the subgroup of $L^{\vee}$ 
\[L+J^\bot_{L^\vee}=\{\mu\in L^\vee | \exists\nu\in L \,\textrm{such that}\, (\mu,\lambda)=(\nu,\lambda)\,\forall \lambda\in J \}.\]

Let $\Theta_D$ denote the vector-valued Theta function associated to $D$ defined by \[\Theta_D(\tau)=\sum_{\beta\in D^\vee} q^{Q(\beta)}e_{\beta+D}\in\C[\Delta_D][[q^{\frac{1}{|\Delta_D|}}]].\]

It is an element of $M_{\frac{b}{2}-1}(\rho_D)$, which is the space of vector-valued modular forms of weight $\frac{b}{2}-1$ with respect to the Weil representation $\rho_D$ associated to the positive definite lattice $(D,Q)$. We can  can also write
\begin{align*}
    \Theta_D(\tau)=\sum_{\beta\in D^\vee/D}\sum_{m\geq 0}c(D,\beta,m)q^m e_{\beta},
\end{align*}
where for $\beta\in D^{\vee}/D$, $m\in Q(\beta)+\Z$, $m\geq 0$, we have \[c(D,\beta,m)=|\{\lambda\in \beta+D, Q(\lambda)=m\}|.\]

Following Bruinier--Zemel's notations \cite[Section 4.4]{bruinierzemel}, define
\begin{align*}
\uparrow^L_D(\Theta_D)(\tau)&=\sum_{\beta\in J^\bot_{L^\vee}/J}q^{\frac{Q(\beta)}{2}}e_{\beta+L}\\
&=\sum_{\beta\in L^\vee/L}\sum_{m\in Q(\beta)+L} c(D,\beta,m)q^m e_{\beta}\in M_{\frac{b}{2}-1}(\rho_L),
\end{align*}
where $c(D,\beta,m)=0$ if $\beta\notin J^\bot_{L^\vee}/J^\bot$ or $m\notin Q(\beta)+\Z$, and otherwise $c(D,\beta,m)=c(D,\overline \beta,m)$ where $\overline \beta$ is the image of $\beta$ under the reduction map $J^{\bot}_{L\vee}\rightarrow D^{\vee}/D$.

In particular, we have 

    \begin{align*}
q\frac{d}{dq}\uparrow^L_D(\Theta_D)(\tau)=\sum_{\beta\in L^\vee/L}\sum_{m\in Q(\beta)+L} mc(D,\beta,m)q^m e_{\beta},
\end{align*}
which is a quasi-modular form in the sense of \cite[Definition 1]{imam}. 

Then by \cite[Definition 4.18, Proposition 4.21 4.15]{bruinierzemel}, we can define:  
\begin{align}\label{Eq:mult-typeII}
    \mu_\Upsilon(m)=\frac{1}{b-2}\mathrm{CT}\left( \langle q\frac{d}{dq}\uparrow^L_D(\Theta_D),F^{+}_m\rangle_L\right),
\end{align}
where $F^{+}_m$ is the holomorphic part of the Harmonic Mass form $F_{m,0}$ from \cite[Proposition 4.2]{bruinierzemel}.
A direct computation shows then (see also the second formula in \cite[Theorem 2.14]{bruinier})
\[ \mu_\Upsilon(m)=\frac{2}{b-2}mc(D,0,m).\]

Classical estimates on coefficients of modular forms, see for example \cite[Prop.1.5.5]{sarnak}, show that \begin{align}\label{equ:growth_D}
    |c(D,\beta,m)|\ll_\epsilon m^{\frac{b}{2}-2+\epsilon}
\end{align}
for all $\epsilon >0$.
Hence we get that 
\[|\mu_\Upsilon(m)|\ll_\epsilon m^{\frac{b}{2}-1+\epsilon},\]
which proves the first part of \Cref{p:growth-multi}.

\subsection{Estimates on type III multiplicities} 
In this section, we prove the estimates on the type III multiplicities in \Cref{p:growth-multi}. 
\medskip

Let $(\Xi,\sigma)$ be a toroidal stratum representative of type III such that $\sigma$ is a ray. Keeping the notations from \Cref{s:typeIII}, let $I_\Q$ be the isotropic line of $L_\Q$ whose stabilizer is the parabolic subgroup attached to $\Xi$ and let $I=I_\Q\cap h.L$. To simplify notations, we assume that $h.L=L$, the reader may notice that this is harmless, up to replacing $L$ by $h.L$ in what follows. 

The line $I$ is an isotropic line of $L$ and the lattice $K_I=I^{\bot}/I$ is Lorentzian. Let $C_\R$ be the cone of negative elements of the Lorentzian space $K_{I,\R}$ and let $C=C_\R\cap K$. As is explained in \Cref{s:typeIII}, the ray $\omega$ is generated by an element $\omega\in K_I\cap C$ which is primitive and such that $Q(\omega)=-N$. Following \cite[Definition 4.18]{bruinierzemel}, we define
\begin{align}\label{Eq:mult-typeIII}
    \mu_{\Xi,\omega}(m)=\frac{\sqrt{N}}{8\sqrt{2}\pi}\Phi^{K}_{m}\left(\frac{\omega}{\sqrt{N}}\right).
\end{align}
Let $v=\frac{\omega}{\sqrt{N}}$. By \cite[Proposition 2.11]{bruinier} and \cite[Theorem 2.14]{bruinier}, we have

\begin{align*}
    \Phi^{K}_{m}(v)&=\Phi^{K}_{m}(v,\frac{1}{2}+\frac{b}{4})\\
    &=\frac{2\Gamma(\frac{b-1}{2})(4\pi m)^{\frac{b}{2}}}{1+\frac{b}{2}}\sum_{\underset{Q(\lambda)=m}{\lambda\in K_I}}\frac{F(\frac{b-1}{2},1,1+\frac{b}{2};\frac{m}{q(Q(\lambda_{v^{\bot}})})}{(4\pi|Q(\lambda_{v^{\bot}})^{\frac{b-1}{2}}|)},
\end{align*}
where $F(a,b,c;z)$ is the usual Gauss hypergeometric function given by:
\[F(a,b,c;z)=\sum_{n=0}^{\infty}\frac{(a)_n(b)_n}{(c)_n}\frac{z^n}{n!},\]
and  $(a)_n=\Gamma(a+n)/\Gamma(a)$. Recall that the above series has $1$ as a radius of convergence and converges absolutely in the unit circle $|z|=1$ if $\mathcal{R}(c-a-b)>0$. In our situation, the latter quantity is equal to $1+b/2-1-\frac{b-1}{2}=1/2>0$. Hence the series $F(\frac{b-1}{2},1,1+\frac{b}{2};z)$ is globally bounded over the unit disc. For $\lambda\in K$ such that $Q(\lambda)=m$, we have $m=Q(\lambda_{v})+Q(\lambda_{v^{\bot}})$ and $Q(\lambda_{v})\leq 0$, hence $0<m\leq Q(\lambda_{v^{\bot}})$. Hence we get
\begin{align*}
    |\Phi^{K}_{m}(v)|&\ll \sqrt{m}.\sum_{\underset{Q(\lambda)=1}{\sqrt{m}\lambda\in K_I}}\frac{1}{Q(\lambda_{v^{\bot}})^{\frac{b-1}{2}}}\\
    &\ll\sqrt{m}\sum_{N\geq 1}\sum_{\underset{Q(\lambda)=1}{\underset{\sqrt{m}\lambda\in K_I}{ Q(\lambda_v^{\bot})\in[N,N+1[}}}\frac{1}{N^{\frac{b-1}{2}}}.\\
\end{align*}
By \Cref{p:quadratic-form} below, we have 
\[|\{\lambda\in K_{I,\R}, Q(\lambda)=1,\sqrt{m}\lambda\in K,Q(\lambda_v^{\bot})\in[N,N+1[\}|\ll_\epsilon m^{\frac{b}{2}-1+\epsilon}N^{\frac{b}{2}-2}.\]

Hence 
\begin{align*}
    |\Phi^{K}_{m}(v)|&\ll_\epsilon m^{\frac{b-1}{2}+\epsilon}\sum_{N\geq 1}\frac{N^{\frac{b}{2}-2}}{{N^{\frac{b-1}{2}}}}\\
    &\ll_\epsilon m^{\frac{b-1}{2}+\epsilon}\sum_{N\geq 1}\frac{1}{N^{3/2}}.\\
    &\ll_\epsilon m^{\frac{b-1}{2}+\epsilon}.
\end{align*}
which proves the second part of \cref{p:growth-multi}.

\begin{proposition}\label{p:quadratic-form}
Let $m\geq 1$ be an integer and $X>0$ a positive real number. Then
\[|\{\lambda\in K_{I,\R}, Q(\lambda)=1,\sqrt{m}\lambda\in K_I,Q(\lambda_v^{\bot})\in[N,N+1[\}|\ll_\epsilon m^{\frac{b}{2}-1+\epsilon}N^{\frac{b}{2}-2} \]
\end{proposition}
\begin{proof}
Recall that $(K_I,Q)$ is a quadratic lattice of signature $(b-1,1)$ and we have a canonical measure $\mu_\infty$ on the quadric $K_1:=\{x\in K_{I,\R}| Q(x)=1\}$ defined as follows: for $W$ an open subset of $K_{\R}$, let 
\[\mu_{\infty}(W\cap K_1)=\lim_{\epsilon \rightarrow 0}\frac{\mathrm{Leb}\left(\{x\in W,\, |Q(x)-1|<\epsilon\}\right)}{2\epsilon}.\]
Here $\mathrm{Leb}$ is the Lebesgue measure on $K_{\R}$ for which the lattice $K$ is of covolume $1$.
One can then prove that (see for example the proof of \cite[Corollary 4.12]{sstt}):
\[\mu_{\infty}(\{\lambda\in K_1, Q(\lambda_{v^{\bot}})\in[X,X+1[\})\ll X ^{\frac{b}{2}-2}.\]
On the other hand, by the equidistribution of integral points in quadrics, see \cite{eskinoh,duke-invetiones-hyperbolic}, \footnote{Or the circle method.} we have: 
\[|\{\lambda\in K_1, \sqrt{m}\lambda\in K,Q(\lambda_v^{\bot})\in[N,N+1[\}|\ll_\epsilon  m^{\frac{b}{2}-1+\epsilon}\mu_{\infty}(\{\lambda\in K_1, Q(\lambda_{v^{\bot}})\in[X,X+1[\}),\]

which yields the desired result. 
\end{proof}

\section{Bounding the contribution from bad reduction places}
In this section we prove \Cref{p:finite-place} and \Cref{t:local-bad-ff}. Let $\mathcal{M}^{\Sigma}$ be as before the toroidal compactification of the GSpin Shimura variety associated to a quadratic lattice $(L,Q)$ and a $K$-admissible polyhedral cone decomposition $\Sigma$. The lattice $(L,Q)$ is assumed to be maximal in the number field case and moreover self-dual at $p$ in the function field case.

\subsection{Bad reduction in the number field setting}
In this section, we prove \Cref{p:finite-place}. We assume hence that the lattice $(L,Q)$ is maximal and that the polyhedral cone decomposition $\Sigma$ is chosen in such way that \Cref{choice-toroidal-nf} is satisfied. 
\medskip 

By the choice of the cone decomposition $\Sigma$, the intersection points of $\cY$ and $\mathcal{M}^{\Sigma}$ lie either in a boundary divisor of type II or a boundary divisor of type III associated to a toroidal stratum representative $(\Xi,\sigma)$ of type III where $\sigma$ is a ray.   

Let $\mathfrak{P}$ be a prime of bad reduction, i.e., where $\mathcal{Y}$ intersects the boundary of $\mathcal{M}^{\Sigma}$. Let $K_\mathfrak{P}$ be the completion at $\mathfrak{P}$ of the number field $K$ and $v_\mathfrak{P}$ its normalized valuation. Let $k_\mathfrak{P}$ be the residue field of $\mathfrak{P}$ and $\overline{k}_\mathfrak{P}$ an algebraic closure.

\subsubsection{Type II degeneration}\label{s:typeII-bad-nf}
Assume in this section that the boundary point lies in $\cB^{\Upsilon}_{\mathbb{F}_p}$ where $\Upsilon$ is a cusp label representative of type II. 

Let $J$ be the primitive isotropic plane associated to $\Upsilon$ and let $D=J_L^{\bot}/J$, see \Cref{s:esti-typeII} for notations.

Recall from \Cref{eq:iso-completions} and \Cref{diagram-typeII} that the completion of $\mathcal{M}^{\Sigma}$ along the boundary divisor $\cB^{\Upsilon}$ fits into the following commutative diagram 
\[
\xymatrix{
  \widehat{\mathcal{M}}_\Upsilon\ar[rr]^{\pi} \ar[d]^{=} && \widehat{\mathcal{M}^{\Sigma}}  
&& \\ 
\widehat{\mathcal{M}}_\Upsilon\ar[rr]^{\mathrm{Spf}(\Z_p[[X]])} && \overline{\mathcal{M}}_{\Upsilon}\ar[rr]^{D\otimes \mathcal{E}}&& \mathcal{M}^{h}_{\Upsilon}
}
\]
where the map $\pi$ is an \'etale map of formal Deligne--Mumford stacks. 
\medskip

The formal completion of $\cY$ along $\mathfrak{P}$ induces a map 
\[\mathrm{Spf}(\mathcal{O}_{K_{\mathfrak{P}}})\rightarrow\widehat{\mathcal{M}^{\Sigma}}, \]
which lifts by \'etaleness of $\pi$ to a map
\[\mathrm{Spf}(\mathcal{O}_{K_{\mathfrak{P}}})\rightarrow\widehat{\mathcal{M}}_\Upsilon.\] 
Denoting by $x$ the image of the closed point $s_{\mathfrak{P}}$, then we get a map of local rings
\[\Psi:\widehat{\mathcal{O}}_{\widehat{\mathcal{M}}_\Upsilon,x}\rightarrow \mathcal{O}_{K_{\mathfrak{P}}}.\]

Let $m \geq 1$ be an integer coprime to $N$. By \Cref{ss:dia-sp-II}, the formal completion of the divisor $\mathcal{Z}(m)$ is  described  as the union over $\lambda\in D$ with $Q(\lambda)=m$, of the vanishing set of the ideals $\Z_p[X]]\otimes \widehat{I}_\lambda$. If $f_\lambda$ is a generator of $\widehat{I}_\lambda$\footnote{Recall that $\cZ(m)$ is Cartier.}, then the multiplicity of intersection of the branch parameterized by $\lambda$ at $\mathfrak{P}$ is equal to 
\[v(\lambda)=v_\mathfrak{P}(\Psi(f_\lambda)).\]
Hence the multiplicity of intersection of $\cY$ and $\cZ(m)$ at $\mathfrak{P}$ is given by: 
\[(\cY.\cZ(m))_{\mathfrak{P}}=\frac{1}{d}\sum_{\underset{Q(\lambda)=m}{\lambda\in D}}v(\lambda),\]
where $d$ is the degree of $\pi$ at $\rho(s_{\mathfrak{P}})$.

For an integer $n$, define the set: 
\[L_n=\{\lambda\in D| v(\lambda)\geq n\},\]
and notice that $(L_n)$ is a decreasing chain of sets. 
It follows then
\begin{align}
    (\cY.\cZ(m))_{\mathfrak{P}}&\leq\sum_{\underset{Q(\lambda)=m}{\lambda\in D}}v(\lambda)\\
    \leq &\sum_{n\geq 1}|\{\lambda\in L_n| Q(\lambda)=m\}|.
\end{align}

The Proposition below should be compared to what happens in the good reduction case in \cite[Section 7]{sstt}. For a definition of the successive minima used, we refer to \cite[Definition 2.2]{EK}.
\begin{proposition}\label{p:lattices-typeII}
The sequence $(L_n,Q)_n$ is a decreasing sequence of  positive definite lattices which all have the same rank $r\leq b-2$. Moreover, the following holds: 
\begin{enumerate}
    \item $\cap_n L_n=\{0\}$.
    \item For every $n\geq 1$, $pL_{n}\subseteq L_{n+1}$.
    \item For $1\leq r\leq b-2$, let $\mu_i(L_n)$ be the $i^{th}$ successive minima of $L_n$ and let $a_i(L_n)=\prod_{1\leq k\leq i}\mu_i(L_n)$. Then we have \[a_i(L_n)\gg_\epsilon n^{\frac{i}{b+\epsilon}}.\]
\end{enumerate}
\end{proposition}
\begin{proof}
Let $\lambda,\lambda'\in L_n$. From \Cref{kernel}, we see that
$\mathrm{ker}(p_\lambda)\cap \mathrm{ker}(p_{\lambda'})$ and thus 
\[\widehat{I}_{\lambda+\lambda}\subset \widehat{I_\lambda}+ \widehat{I_{\lambda'}}.\]

It follows that \[v(\lambda+\lambda')\geq \min\{v(\lambda),v(\lambda')\}\geq n.\] We conclude that $L_n\subseteq D$ is a subgroup and $(L_n,Q)$ is obviously positive definite. Moreover, since the curve $\cY$ is not contained in any special divisor, $(1)$ follows immediately. 

For (2), let $\lambda\in L_n$ with $v(\lambda)\geq n\geq 1$.  Then $\widehat{I}_{p\lambda}$ is the ideal defining the kernel of the composition 
\[D\otimes\widehat{\mathcal{E}}\rightarrow \widehat{\mathcal{E}}\rightarrow \widehat{\mathcal{E}},\]
over $\mathrm{Spf}(\widehat{\mathcal{O}}_{\mathcal{M}^h,z})$.

Since the multiplication by $p$ map is ramified at $0$ with ramification degree equal to $p$, we conclude that 
\begin{align*}
    v(p\lambda)&\geq pv(\lambda)\\
    &\geq n+1.
\end{align*}
This also proves that the lattices $L_n$ have the same rank.

For (3), let $n\geq 1$ and let $w_0$ be a vector in $L_n$ such that $Q(w_0)=\mu_1(L_n)^2$. By choosing $m_0=\mu_1(L_n)^2$, the height bound \Cref{Diopbound} implies 
\[n\leq (\cY.\cZ(m_0))_\mathfrak{P}\ll_\epsilon m^{\frac{b+\epsilon}{2}}.\]
Hence $\mu_1(L_n)\gg_\epsilon n^{\frac{1}{b+\epsilon}}$. Since $a_i(n)\geq \mu_1(n)^r$, this concludes the proof.
\end{proof}
\begin{proposition}\label{p:lattice-estimate}
Let $D\in \Z_{\geq 1}$. For $X\in \Z_{>0}$, let $S_{D,X}$ denote the set \[\{m\in \Z_{>0}\mid X \leq m<2X,\, \frac{m}{D}\in \Z \cap (\Q^\times)^2,\, (m,N)=1\}.\] Then we have
\[\sum_{m\in S_{D,X}}(\cY . \cZ(m))_\fP=o(X^{\frac{b+1}{2}}\log X).\]
\end{proposition}
\begin{proof}
We have 
\begin{align*}\sum_{m\in S_{D,X}}(\cY.\cZ(m))_\fP&\leq\sum_{m\in S_{D,X}}\sum_{n\geq1}|\{\lambda\in L_n|Q(\lambda)=m\}|\\
&=\sum_{n\geq1}\sum_{m\in S_{D,X}}|\{\lambda\in L_n|Q(\lambda)=m\}|.
\end{align*}
By \cite[Lemma 2.4]{EK}, we have the following estimate which only depends on the rank $r$ of the lattices $L_n$ and hence not on $n$
\[\sum_{m\in S_{D,X}}|\{\lambda\in L_n|Q(\lambda)=m\}|\ll \sum_{j=0}^{r}\frac{X^{j}}{a_j(L_n)}.\]
On the other hand, if $\lambda\in L_n$ with $Q(\lambda)=m\in S_{D,X}$, then $\mu_1(L_n)^2\leq m\leq X$, hence $n\ll X^{\frac{b+\epsilon}{2}}$ and
\begin{align*}
    \sum_{m\in S_{D,X}}(\cY.\cZ(m))_\fP&\ll\sum_{m\in S_{D,X}}\sum^{O_\epsilon(X^{\frac{b+\epsilon}{2}})}_{n\geq1}|\{\lambda\in L_n|Q(\lambda)=m\}|\\
&\ll\sum_{j=0}^{r}\sum_{n\geq1}^{X^{\frac{b+\epsilon}{2}}}  \frac{X^{\frac{j}{2}}}{n^{\frac{j}{b+\epsilon}}}\\
 &\ll\sum_{j=0}^{r} X^{\frac{j}{2}+(1-\frac{j}{b+\epsilon})\frac{b+\epsilon}{2}}\\
&=O(X^{\frac{b+\epsilon}{2}}).
\end{align*}
Hence the result. 
\end{proof}
\subsubsection{Type III degeneration}\label{s:typeIII-bad-red}
Let $(\Xi,\sigma)$ be a toroidal stratum representative of type III such that $\sigma$ is a ray. We use notations from \Cref{s:typeIII}.
\medskip

By our choice of $\Sigma$, the  curve $\cY$ touches the boundary of $\mathcal{M}^{\Sigma}$  at a locally closed boundary divisor $\cB^{\Xi,\sigma}$. Let $\widehat{M}^{\Sigma}$ be the formal completion of $\mathcal{M}^{\Sigma}$ along $\cB^{\Xi,\sigma}$ and hence we get a map 
\begin{align}\label{dia-formal-tori}
    \widehat{\cY}\rightarrow \widehat{M}^{\Sigma}.
\end{align}
By \Cref{ss:fomral-tIII}, the following maps of formal Deligne--Mumford stacks are finite \'etale: 

\[ \bigsqcup_{\Q^{\times}_{>0}\backslash \mathbb{A}_f^{\times}/K_0} \widehat{\mathcal{T}}_{\Xi/R}\rightarrow\widehat{\mathcal{M}}_{\Xi,\sigma}\rightarrow \widehat{\mathcal{M}^{\Sigma}}.\]

Hence map \ref{dia-formal-tori} lifts to map 
\[\widehat{\cY}\rightarrow\mathrm{Spf}\left(\Z_p[q_\alpha|\alpha\in\Gamma_\Xi^{\vee}\cap \omega^{\bot}][[q_{\omega'}]]\right).\]





This corresponds to a morphism: 
\begin{align}\label{eq:typeIII-nf}
\Z_{(p)}[[q_{\omega'}]][q_\alpha]_{\underset{(\alpha.\omega)=0}{\alpha\in\Gamma_\Xi^{\vee}}}\rightarrow \mathcal{O}_{K_{\mathfrak{P}}}.
\end{align}

Let $\lambda\in K_I=\Gamma_\Xi$ with $Q(\lambda)=m$. By \Cref{ss:special-tIII} the branch of the special divisor $\cZ(m)$ parameterized by $\lambda$ intersects the boundary only if $ (\lambda.\omega)=0$. In the latter case, by \Cref{equation-typeIII},  its equation is given by $q^{\lambda}-1$ and the multiplicity of intersection of $\cY$ with the branch given by $\lambda$ is the $p$-adic valuation of the element $q^{\lambda}-1$ under the map \Cref{eq:typeIII-nf}.

Let $x\in \cB^{\Xi,\sigma}(\overline{\F}_p)$ be the image of $\mathfrak{P}$. Then by the previous discussion, we conclude that \[(\cY.\cZ(m))_\mathfrak{P}=\frac{1}{d}\sum_{\underset{Q(\lambda)=m}{\lambda\in K_I\cap \omega^{\bot}}}v_p(q^{\lambda}-1),\]
where $d$ is the degree of the map \Cref{dia-formal-tori} at $x$.

For $n\geq 1$, let 
\[L_n=\{\lambda \in K_I\cap \omega^{\bot}| v_p(q^{\lambda}-1)\geq n\}.\]

Then we can rewrite the multiplicity intersection at $\mathfrak{P}$ as: 
\[(\cY.\cZ(m))_{\fP}=\frac{1}{d}\sum_{n\geq 1}\{\lambda\in L_n| Q(\lambda)=m\}.\] 
\begin{proposition}\label{p:lattice-decayIII}
The lattices $(L_n,Q)$ are positive definite lattices of rank $r\leq b-1$ independent from $n$ and they satisfy the following properties: 
\begin{enumerate}
    \item $\cap_n L_n=\{0\}$.
    \item For every $n\geq 1$, $pL_{n}\subseteq L_{n+1}$.
    \item For  $1\leq r\leq b-1$, let $\mu_i(L_n)$ be the $i^{th}$ successive minima and let $a_i(L_n)=\prod_{1\leq k\leq i}\mu_i(L_n)$. Then we have \[a_i(L_n)\gg_\epsilon n^{\frac{i}{b+\epsilon}}.\]
\end{enumerate}
\end{proposition}
\begin{proof}
The proof is similar to the proof of \Cref{p:lattices-typeII}. Let $\lambda,\lambda' \in K\cap\omega^{\bot}$. 
By writing \[q^{\lambda+\lambda'}-1=q^{\lambda}(q^{\lambda'}-1)+q^{\lambda}-1,\]
we get that $L_n$ is a lattice and it is obviously positive definite as $K_I$ is Lorentzian and $\omega$ is a negative normed vector. 

Let $\pi$ be a uniformizer of $\mathcal{O}_{K_\fP}$ and let $\lambda\in L_n$. Then $q^{\lambda}=1+\pi^n.u$ for some $u\in \mathcal{O}_{K_\fP}$. Hence 
\begin{align*}
    q^{p\lambda}-1&=(1+\pi^n.u)^p-1\\
    &=\sum_{i\geq 1}\binom{p}{i}\pi^{ni}u^i\\
    &=\pi^{n+1}u'.
\end{align*}
Hence (2). The rest of the proof is similar to \Cref{p:lattices-typeII}.
\end{proof}
As a consequence, we get the following proposition, whose proof is identical to the proof of \Cref{p:lattice-estimate} and we omit it.
\begin{proposition}\label{p:lattice-estimate-III}
Let $D\in \Z_{\geq 1}$ be coprime to $N$. For $X\in \Z_{>0}$, let $S_{D,X}$ denote the set \[\{m\in \Z_{>0}\mid X \leq m<2X,\, \frac{m}{D}\in \Z \cap (\Q^\times)^2,\,(m,N)=1\}.\] Then we have
\[\sum_{m\in S_{D,X}}(\cY . \cZ(m))_\fP=o(X^{\frac{b+1}{2}}\log X).\]
\end{proposition}


\subsection{Function field setting}
In this section, we prove \Cref{p:local-good-ff}. We assume here that the lattice $(L,Q)$ is self-dual at $p$ and we let $\mathcal{M}_{\F_p}$ be the mod $p$ GSpin Shimura variety associated to $(L,Q)$. Let $\Sigma$ be a polyhedral cone decomposition which satisfies \Cref{p:choice-tor-ff}. 
\medskip

Let $\overline{\scrS}\rightarrow \mathcal{M}^{\Sigma}_{\mathbb{F}_p}$ be a finite map as before and let $P\in \overline{\scrS}(\overline{\F}_p)$ be a point mapping to the boundary of $\mathcal{M}^{\Sigma}_{\mathbb{F}_p}$. Let denote $k=\overline{\F}_p$. The point $P$ lies either in a boundary stratum of of type II or type III. We treat each case separately.

\subsubsection{Type II degeneration} 

Assume that the image of $P$ is in $\cB^{\Upsilon}_{\mathbb{F}_p}(k)$ where $\Upsilon$ is a cusp label representative of type II. 

Let $\widehat{\overline{\mathscr S}}\simeq\mathrm{Spf}(k[[t]])$ be the formal completion of $\overline{\scrS}$ along $s$. Then by reasoning similarly to \Cref{s:typeII-bad-nf}, specifically using the reduction mod $p$ of \Cref{diagram-typeII}, we get for every $\lambda\in D$ with $Q(\lambda)=m\geq 1$, $m$ coprime to $N$ a map 

\[\Phi_p:\widehat{\mathcal{O}}_{\mathcal{M}_{\Upsilon,\mathbb{F}_p,x}}\rightarrow k[[t]],\]

Let $v(\lambda)$ denote the $t$-adic valuation of the generator $f_\lambda$ of $I_{\lambda,p}$. Then similarly to the number field case, we have: 
\begin{lemma}
The multiplicity of intersection of $\overline{\scrS}$ and $\cZ(m)_{\F_p}$ at $P$ satisfies: 
\[m_P(\overline\scrS,\cZ(m)_{\F_p})\ll\sum_{n\geq 1}|\{\lambda\in L_n|\,Q(\lambda)=m\}|.\]
\end{lemma}

Now we are ready to prove \Cref{t:local-bad-ff}.

\begin{proposition}
The sequence of lattices $(L_n,Q)$ satisfy the same properties as in \Cref{p:lattices-typeII} and letting $S$ be as in Section 4.2, we have the following estimate for $X>0$: 
\[\sum_{m\in S_X}m_P(\overline\scrS,\cZ(m)_{\F_p})=O_\epsilon(X^{\frac{b+\epsilon}{2}})\]
\end{proposition}
\begin{proof}
The same proof as in \Cref{p:lattices-typeII} shows that the lattices $(L_n,Q)$ enjoy the same properties of the aforementioned proposition. For the second part, we have
\begin{align*}
    \sum_{m\in S_X}m_P(\overline\scrS,\cZ(m)_{\F_p})&\ll\sum_{m\in S_X}\sum_{n\geq 1}|\{\lambda\in L_n|\,Q(\lambda)=m\}|\\
    &\ll\sum^{O(X^{\frac{b+\epsilon}{2}})}_{n=1} |\{\lambda\in L_n| Q(\lambda)\leq m\}|\\
    &\ll\sum^{O(X^{\frac{b+\epsilon}{2}})}_{n=1}\sum_{j=0}^{r}\frac{X^{\frac{j}{2}}}{a_j(L_n)}\\
    &\ll \sum_{j=0}^{r}\sum^{O(X^{\frac{b+\epsilon}{2}})}_{n=1}\frac{X^{\frac{j}{2}}}{n^{\frac{j}{b+\epsilon}}}\\
    &=O(X^{\frac{b+\epsilon}{2}}).
\end{align*}
\end{proof}
\subsubsection{Type III degeneration}
Assume now that there exists a toroidal stratum representative $(\Xi,\sigma)$ such that $\sigma$ is a ray and such that $P$ lies in $\cB^{\Xi,\sigma}_{\mathbb{F}_p}(k)$. 
Using a similar approach to \Cref{s:typeIII-bad-red} by taking reduction mod $p$, we get a map

\[k[q_\alpha|\,\alpha\in\Gamma_\Xi^{\vee}\cap \omega^{\bot}][[q_{\omega'}]]\rightarrow k[[t]],\]
sending $q_{\omega'}$ to an element of the ideal $(t)$.
Let $v$ denote the $t$-adic valuation on $k[[t]]$. Then, for $m$ coprime to $N$, the multiplicity of intersection of $\overline{\scrS}$ and $\cZ(m)_{\F_p}$ at $P$ satisfies: 
\[m_P(\overline\scrS,\cZ(m)_{\F_p})\leq\sum_{\underset{Q(\lambda)=m}{\lambda\in K_I\cap \omega^{\bot}}}v(q^{\lambda}-1).\]
If we define the sequence lattices $L_n$ as 
\[L_n=\{\lambda\in K\cap \omega^{\bot}|v(q^\lambda-1)\geq n,\}\]
then
\[m_P(\overline\scrS,\cZ(m)_{\F_p})\leq \sum_{n\geq 1}|\{\lambda\in L_n,\,Q(\lambda)=m\}|.\]

Now the rest of the proof is similar to \Cref{s:typeIII-bad-red}. This proves \Cref{t:local-bad-ff} in the remaining type III case.

\section{Applications}
In this section, we present a proof of \Cref{t:hecke-orbit}. This approach is inspired from \cite{maulik-shankar-tang-K3}.
\subsection{Hecke orbit conjecture}
\subsubsection{The orthogonal case}
Since GSpin Shimura varieties are finite covers of orthogonal ones, it is enough to prove the result for GSpin Shimura varieties.

Let $\mathcal{M}_{\F_p}$ be the reduction mod $p\geq 5$ of a GSpin type Shimura variety with hyperspecial level at $p$ associated to a lattice $(L,Q)$, which is assumed to be self-dual at $p$ and of signature $(b,2)$.  We will  prove \Cref{t:hecke-orbit} by induction on $b$, which is also the dimension of $\mathcal{M}_{\F_p}$.

The case $b=1$ is immediate: the prime-to-$p$ Hecke orbit of $x$ is infinite, hence Zariski dense. 
\medskip 

Assume now that $n\geq 2$ and the result of \Cref{t:hecke-orbit} holds for all ordinary points in GSpin Shimura varieties of dimension less than $b-1$ with hyperspecial level at $p$. Let $x$ be an ordinary point in $\mathcal{M}(\overline{\F}_p)$ and let  $\overline{T_x}$ be the Zariski closure of its prime-to-$p$ Hecke orbit. Then $\overline{T_x}$ has positive dimension and intersects the ordinary locus non-trivially. Hence we can find a smooth  quasi-projective curve $\scrS$ and a finite map 
\[\scrS\rightarrow \mathcal{M}_{\mathbb{F}_p}\] whose image is contained in  $\overline{T_x}$ and which is contained in the ordinary locus. Moreover, we can assume that this image is not contained in any special divisor. Indeed, the same argument used for proper curves in \cite[Lemma 8.11]{maulik-shankar-tang-K3} works in our setting with no change. By \Cref{th:K3ff}, the curve $\scrS$ intersects infinitely many divisors $\cZ(m)_{\F_p}$ with $(m,p)=1$. The special divisors $\cZ(m)_{\F_p}$ are themselves the union of GSpin Shimura varieties of dimension $b-1$ with hyperspecial level at $p$ since $m$ is coprime to $p$. Let $y\in\scrS(\overline{\mathbb{F}}_p)\cap\cZ'(m)(\overline{\F}_p)$ for some irreducible component $\cZ'(m)$ of $\cZ(m)$. Then $y$ is ordinary and the prime-to-$p$ Hecke orbit of $y$ in $\cZ'(m)_{\F_p}$ is Zariski dense by the induction hypothesis. Since this orbit is a sub-orbit of the Hecke orbit in $\mathcal{M}_{\F_p}$, we conclude that $\cZ'(m)_{\F_p}\subset \overline{T_x}$. Furthermore, it is straightforward to check that the collection of the divisors $\cZ'(m)_{\F_p}$ must be infinite by \Cref{th:K3ff}. Hence we conclude that $\overline{T_x}=\mathcal{M}_{\F_p}$ which is the desired result. 
\subsubsection{The unitary case}
We prove in this section the Hecke orbit conjecture in the unitary case using the reduction to the orthogonal case already used in \cite[Remark 8.12]{maulik-shankar-tang-K3} and in \cite[Section 9.3]{sstt}.

Let $\mathcal{M}_{\F_p}$ be the mod $p$ points of the canonical model of a unitary Shimura variety associated to an imaginary quadratic field $\mathbf{k}$, a unitary group of signature $(r,1)$ with hyperspecial level at $p$ as described in \cite[Section 2.1]{bruinier-howard} such that $p$ is split in $\mathbf{k}$. Consider the family of special divisors $\cZ_{Kra}(m)$ as described in Section 2.5 {\it loc.} {\it cit.} which are themselves unitary Shimura varieties associated to unitary groups of signature $(r-1,1)$ and hyperspecial at $p$ when $p$ does not divide $m$. Then using a similar argument to \cite[Section 9.3]{sstt} and further explained in \cite[Remark 8.12]{maulik-shankar-tang-K3}, we have the following theorem which is a consequence of \Cref{th:K3ff}. 

\begin{theorem}\label{unitary}
Assume that $p\geq 5$ and let $\mathscr{S}\rightarrow \mathcal{M}_{\F_p}$ be a finite map from a smooth quasi-projective curve $\mathscr{S}$ over $\overline{\mathbb{F}}_p$ and with generically ordinary image. Then the union over $m$ prime to $p$ of the intersections $\scrS\cap \cZ_{Kra}(m)$ is infinite.
\end{theorem}

Now the Hecke orbit conjecture in the unitary case is an easy consequence of the above theorem and the induction method explained in the previous paragraph. 
\bibliographystyle{alpha}
\bibliography{bibliographie}
\end{document}